\documentclass[12pt,reqno]{amsart}
\usepackage[activeacute,english]{babel}

\usepackage[utf8]{inputenc}
\usepackage{amsmath} 
\usepackage{amsthm} 
\usepackage{amssymb} 
\usepackage{amscd} 

\usepackage{emptypage}
\usepackage{enumitem} 
\usepackage{mathtools} 
\usepackage{mathrsfs} 
\usepackage{upgreek}

\usepackage{hyperref,cleveref,autonum}

\usepackage{esint} 
\usepackage{graphicx,caption} 
\usepackage{float} 

\usepackage{pgf,tikz,pgfplots}
\usetikzlibrary{arrows}
\usetikzlibrary{intersections}
\usetikzlibrary{fadings}
\usepackage{wrapfig}

\usepackage[margin=0.9in]{geometry}
\newtheorem{theorem}{Theorem}[section]
\newtheorem*{theorem*}{Theorem}

\newtheorem{lemma}[theorem]{Lemma}
\newtheorem{corollary}[theorem]{Corollary}

\theoremstyle{definition}

\theoremstyle{remark}
\newtheorem{remark}[theorem]{Remark}
\newtheorem*{remark*}{Remark}

\newcommand{\con}[1]{\mathbb{#1}}
\newcommand{\C}{\con{C}} 
\newcommand{\R}{\con{R}} 

\newcommand{\Dirac}{\mathcal{H}}
\newcommand{\D}{\mathcal{D}}
\newcommand{\re}{\operatorname{Re}}
\newcommand{\im}{\operatorname{Im}}

\numberwithin{equation}{section}
\title[Convergence of generalized MIT bag models to zigzag  Dirac operators]
{Convergence of generalized MIT bag models to\\
Dirac operators with zigzag boundary conditions}

\author[J. Duran]{Joaquim Duran}
\address{J. Duran
\newline
Centre de Recerca Matem\`atica, Edifici C, Campus Bellaterra, 08193 Bellaterra, Spain.}
\email{jduran@crm.cat}

\author[A. Mas]{Albert Mas}
\address{A. Mas \textsuperscript{1,2}
\newline
\textsuperscript{1}
Departament de Matem\`atiques,
Universitat Polit\`ecnica de Catalunya,
Campus Diagonal Bes\`os, Edifici A (EEBE), Av. Eduard Maristany 16, 08019
Barcelona, Spain
\newline
\textsuperscript{2}
Centre de Recerca Matem\`atica, Edifici C, Campus Bellaterra, 08193 Bellaterra, Spain.}
\email{albert.mas.blesa@upc.edu}

\date{\today}

\thanks{The two authors are supported by the Spanish grants PID2021-123903NB-I00 and RED2018-102650-T funded by MCIN/AEI/10.13039/501100011033 and by ERDF ``A way of making Europe'', and by the Catalan grant 
2021-SGR-00087. 
This work is supported by the Spanish State Research Agency, through the Severo Ochoa and Mar\'ia de Maeztu Program for Centers and Units of Excellence in R\&D (CEX2020-001084-M). The first author is supported by CEX2020-001084-M-20-1 and acknowledges CERCA Programme/Generalitat de Catalunya for institutional support.}

\begin{document}

\begin{abstract}
This work addresses the resolvent convergence of  generalized MIT bag operators to Dirac operators with zigzag type boundary conditions. We prove that the convergence holds in strong but not in norm resolvent sense. Moreover, we show that the only obstruction for having norm resolvent convergence is the existence of an eigenvalue of infinite multiplicity for the limiting operator. More precisely, we prove the convergence of the resolvents in operator norm once projected into the orthogonal of the corresponding eigenspace.
\end{abstract}

\maketitle

\section{Introduction}

In the 1970s the MIT bag model was introduced as a simplified three-dimensional model to describe the confinement of relativistic particles in a box, such as quarks in hadrons \cite{bogoliubov1968,johnson}. From a mathematical perspective, a family of Dirac operators $\{\Dirac_\tau\}_{\tau\in\R}$ which includes the MIT bag model was investigated in \cite{Mas2022}.\footnote{A two-dimensional analog of the family $\{\Dirac_\tau\}_{\tau\in\R}$ is used in the mathematical description of graphene; see \cite{Benguria2017Self,Benguria2017Spectral} and the references therein. In \Cref{ss:quantum_dots} we will comment on how the results of the present paper also adapt to the two-dimensional setting.}
In that work it was pointed out the interest of studying the convergence, as $\tau \to \pm \infty$, of the operators $\Dirac_\tau$ to the Dirac operators with zigzag boundary conditions investigated in \cite{Holzmann2021}. The purpose of the present work is to perform such study.

\subsection{Generalized MIT bag models and zigzag boundary conditions}
Let $-i \alpha \cdot \nabla + m \beta$ denote the differential expression that gives the action of the free Dirac operator in $\R^3$. Here,
$m\geq0$ denotes the mass, 
$\nabla=(\partial_1,\partial_2,\partial_3)$ denotes the gradient in $\R^3$, $\alpha := (\alpha_1, \alpha_2, \alpha_3)$, 
$$
\alpha_j:=\begin{pmatrix}0&\sigma_j\\\sigma_j&0\end{pmatrix} \mbox{ for }j=1,2,3, \quad  \text{and}\quad
\beta:=\begin{pmatrix}I_2&0\\0&-I_2\end{pmatrix}	
$$
are the $\C^{4\times4}$-valued Dirac matrices,
$I_2$ denotes the identity matrix in $\C^{2\times 2}$, and
$$
\sigma_1 := 
\begin{pmatrix}
	0&1\\1&0
\end{pmatrix},
\quad
\sigma_2 := 
\begin{pmatrix}
	0&-i\\i&0
\end{pmatrix},
\quad
\sigma_3 := 
\begin{pmatrix}
	1&0\\0&-1
\end{pmatrix}
$$
are the Pauli matrices. As customary, we use the notation $\alpha \cdot X:=\alpha_1 X_1+\alpha_2 X_2+\alpha_3 X_3$ for $X = (X_1,X_2,X_3)$, and analogously for $\sigma \cdot X$ with
$\sigma := (\sigma_1,\sigma_2,\sigma_3)$.

Let $\Omega\subset \R^3$ be a bounded domain with $C^2$ boundary. We denote by $\nu$ the unit normal vector field at $\partial\Omega$ which points outwards of $\Omega$. Given $\tau\in\R$, let $\Dirac_\tau$ be the Dirac operator in $L^2(\Omega)^4$ defined by
\begin{equation}\label{def:Dirac_tau}
	\begin{split}
		\mathrm{Dom}(\Dirac_\tau) &:= \big\{ \varphi \in H^1(\Omega)^4: \, \varphi = i (\sinh\tau- \cosh\tau \, \beta)( \alpha \cdot\nu ) \varphi  \,\text{ on } \partial \Omega \big\},\\
		\Dirac_\tau\varphi &:= (-i \alpha \cdot \nabla + m \beta)\varphi \quad\text{for all 
			$\varphi\in\mathrm{Dom}(\Dirac_\tau);$}
	\end{split}
\end{equation}
the boundary condition in 
$\mathrm{Dom}(\Dirac_\tau)$ is understood in the $L^2(\partial\Omega)^4$ sense.\footnote{The reader may look at \Cref{ss:notation}, where we recall the basic notation used throughout the paper.}
The set $\{\Dirac_\tau\}_{\tau\in\R}$ was introduced in \cite{Mas2022} as a family of confining models generated by electrostatic and Lorentz scalar $\delta$-shell potentials; the so-called MIT bag model corresponds to $\tau=0$. For 
$\tau\in\R$, the operator $\Dirac_\tau$ is self-adjoint in $L^2(\Omega)^4$ by \cite[Proposition 5.15]{Behrndt2020}. Moreover, from \cite[Lemma 1.2]{Mas2022} we know that its spectrum 
$\sigma(\Dirac_\tau)$ is contained in 
$\R\setminus{[-m,m]}$ and is purely discrete; see \Cref{rk:disc_spec} for further comments on this last assertion. In particular, the essential spectrum $\sigma_{\operatorname{ess}}(\Dirac_{\tau})$ is empty for all $\tau\in\R$.
Furthermore, 
$\lambda\in\sigma(\Dirac_\tau)$ if and only if 
$-\lambda\in\sigma(\Dirac_{-\tau})$.

A spectral study of the mapping 
$\tau\mapsto\Dirac_\tau$ was carried out in \cite{Mas2022}. It was shown that the eigenvalues of $\Dirac_\tau$ can be parametrized by  increasing real analytic functions of $\tau$. Special attention was paid on the asymptotic behavior of the eigenvalues as $\tau\to\pm\infty$.
Thanks to the odd symmetry of the eigenvalues with respect to the parameter $\tau$ mentioned before, this study was reduced to $\sigma(\Dirac_\tau)\cap(m,+\infty)$, and the following result was shown. In its statement, $-\Delta_D$ denotes the self-adjoint realization of the Dirichlet Laplacian in $L^2(\Omega)$. 

\begin{theorem}\label{AMSV:eig_curv}
$($\cite[Theorem 1.4]{Mas2022}$)$
Let $\tau\mapsto\lambda(\tau)\in\sigma(\Dirac_\tau)\cap(m,+\infty)$ be a continuous function defined on an interval $I\subset\R$. The following holds:
\begin{enumerate}[label=$(\roman*)$]
\item If $I=(-\infty,\tau_0)$ for some $\tau_0\in\R$, then 
$\lambda(-\infty):=\lim_{\tau\downarrow-\infty}\lambda(\tau)$ exists and belongs to $[m,+\infty)$. In addition,
\begin{equation}
\lambda(-\infty)=		
\begin{cases}
m  \text{ if $\lambda(\tau)\leq\sqrt{\min\sigma(-\Delta_D)+m^2}$ for some $\tau\in I$,}\\
\sqrt{ \lambda_D+m^2} \text{ for some 
$\lambda_D\in \sigma(-\Delta_D)$} \text{ otherwise.}
\end{cases}
\end{equation}
\item If $I=(\tau_0,+\infty)$ for some $\tau_0\in\R$, then 
$\lambda(+\infty):=\lim_{\tau\uparrow+\infty}\lambda(\tau)$ exists as an element of the set $(m,+\infty]$. 
In addition, if $\lambda(+\infty)<+\infty$ then 
$$\lambda(+\infty)={\textstyle\sqrt{ \lambda_D+m^2}} 
\quad \text{ for some }
\lambda_D\in \sigma(-\Delta_D).$$
\end{enumerate}
\end{theorem}
This result establishes a clear connection between the spectrum of the Dirac operator $\Dirac_\tau$ as $\tau\to\pm\infty$ and the spectrum of the Dirichlet Laplacian $-\Delta_D$. This leads to ask which should be the limiting operators of $\Dirac_\tau$ as $\tau\to\pm\infty$, and to investigate in which resolvent sense the convergence holds true. 

The first question is easily answered thanks to the following observation. Given $\varphi\in\mathrm{Dom}(\Dirac_\tau)$, if we write it in components\footnote{Given a matrix $A$, we denote by $A^\intercal$ the transpose matrix of $A$. The notation $\varphi= (u, v)^\intercal$, to be used throughout this article, refers to the   decomposition of $\varphi:\Omega\to\C^4$ in upper and lower components, that is, if $\varphi=(\varphi_1,\varphi_2,\varphi_3,\varphi_4)^\intercal$ with $\varphi_j:\Omega\to\C$ for $j=1,2,3,4$, then $u=(\varphi_1,\varphi_2)^\intercal$ and $v=(\varphi_3,\varphi_4)^\intercal$.}  as $\varphi = (u, v)^\intercal$, the boundary condition
$$\varphi = i (\sinh\tau- \cosh\tau\, \beta)( \alpha \cdot\nu ) \varphi$$ rewrites in terms of the components as $u = -ie^{-\tau} (\sigma \cdot \nu) v$. Formally, 
this equation forces $u$ and $v$ to vanish on $\partial \Omega$ in the limits $\tau \uparrow +\infty$ and $\tau \downarrow -\infty$, respectively. This leads to consider the Dirac operators with zigzag type boundary conditions studied in \cite{Holzmann2021}, which are defined by
\begin{equation}\label{def:Dirac_inft}
\begin{split}
\mathrm{Dom}(\Dirac_{+\infty}) 
&:= \big\{ \varphi= (u, v)^\intercal : \, u \in H^1_0(\Omega)^2,  v \in L^2(\Omega)^2, \alpha \cdot \nabla \varphi \in L^2(\Omega)^4 \big\},\\
\Dirac_{+\infty}\varphi &:= (-i \alpha \cdot \nabla + m \beta)\varphi \quad\text{for all $\varphi\in\mathrm{Dom}(\Dirac_{+\infty})$}
\end{split}
\end{equation}
and
\begin{equation}\label{def:Dirac_-inft}
\begin{split}
\mathrm{Dom}(\Dirac_{-\infty}) 
&:= \big\{ \varphi= (u, v)^\intercal : \, u \in L^2(\Omega)^2, v \in H^1_0(\Omega)^2, \alpha \cdot \nabla \varphi \in L^2(\Omega)^4 \big\},\\
\Dirac_{-\infty}\varphi &:= (-i \alpha \cdot \nabla + m \beta)\varphi \quad\text{for all $\varphi\in\mathrm{Dom}(\Dirac_{+\infty}).$}
\end{split}
\end{equation}
From \cite[Theorem 1.1 and Lemma 3.2]{Holzmann2021} we know that $\Dirac_{\pm\infty}$ are self-adjoint in $L^2(\Omega)^4$ and that 
\begin{equation} \label{SpectraPMIn}
    \begin{split}
     \sigma(\Dirac_{+\infty}) & = \{-m\} \cup \big\{\pm {\textstyle\sqrt{\lambda_D+m^2}}:\, \lambda_D\in \sigma(-\Delta_D) \big\}, \\
       \sigma(\Dirac_{-\infty}) & = \{m\} \cup \big\{\pm {\textstyle\sqrt{\lambda_D+m^2}}:\, \lambda_D\in \sigma(-\Delta_D) \big\}
    \end{split}
\end{equation}
(observe that this description of $\sigma(\Dirac_{\pm\infty})$ is in agreement with \Cref{AMSV:eig_curv}).
In addition, $\pm m\in \sigma_{\operatorname{ess}}(\Dirac_{\mp\infty})$ is an eigenvalue of infinite multiplicity. The non-emptiness of the essential spectrum implies that $\mathrm{Dom}(\Dirac_{\pm\infty})$ is not contained in $H^1(\Omega)^4$. Thus, the eigenvalue $\pm m$ gives rise to a loss of regularity in $\mathrm{Dom}(\Dirac_{\mp\infty})$ with respect to $\mathrm{Dom}(\Dirac_{\tau})$ for $\tau\in\R$.

The main purpose of this work is to study the convergence of $\Dirac_{\tau}$ to $\Dirac_{\pm\infty}$ as 
$\tau\to\pm\infty$ in strong and norm resolvent senses. This study was motivated in \cite[Remark 4.4]{Mas2022} but it was not addressed throughout that paper; see \Cref{LimitFirstEigenvalue} and the paragraph that precedes it for more details. As we will see, the eigenvalue $\pm m$ will also play an important role in this study.

\subsection{Main results} \label{ss:main_results}

We begin this section with a simple observation which shows that $\Dirac_\tau$ does not converge to 
$\Dirac_{\pm \infty}$ in the norm resolvent sense as $\tau \to \pm \infty$: 
if there was convergence in the norm resolvent sense we would get that 
$\lim_{\tau\to\pm\infty}\sigma_{\operatorname{ess}}(\Dirac_\tau)=\sigma_{\operatorname{ess}}(\Dirac_{\pm\infty})$ by \cite[Satz 9.24]{Weidmann2000}, but this is impossible since $\sigma_{\operatorname{ess}}(\Dirac_{\pm\infty})\neq\emptyset$ and $\sigma_{\operatorname{ess}}(\Dirac_{\tau})=\emptyset$ for all $\tau\in\R$. 

Once norm resolvent convergence is discarded, it is natural to ask whether the convergence in the strong resolvent sense holds or not. The following result answers this question in the affirmative.

\begin{theorem}\label{thm:res_conv_1}
Given $\tau\in\R$, let $\Dirac_\tau$ be the operator defined in \eqref{def:Dirac_tau}. Let $\Dirac_{+ \infty}$ and 
$\Dirac_{-\infty}$ be the operators defined in \eqref{def:Dirac_inft} and \eqref{def:Dirac_-inft}, respectively. Then, 
$\Dirac_\tau$ converges to $\Dirac_{\pm \infty}$ in the strong resolvent sense as $\tau\to\pm\infty$.
\end{theorem} 

Combining this theorem with \cite[Theorem VIII.24 $(a)$]{ReedSimon1980} we get that the spectra of the limiting 
operators $\Dirac_{\pm \infty}$ cannot suddenly expand with respect to the spectrum of $\Dirac_\tau$. The fact that they cannot either suddenly contract was already known by \Cref{AMSV:eig_curv} ---the non-contraction effect is a classical consequence of norm resolvent convergence \cite[Theorem VIII.23 $(a)$]{ReedSimon1980}, but recall that such a convergence does not hold in our setting.

As an application of \Cref{thm:res_conv_1}, we characterize  the asymptotic behavior as $\tau\uparrow+\infty$ of the function 
$\lambda_1^+:\R\to(m,+\infty)$ defined by 
\begin{equation} \label{defi:ParamFirstEig}
	\tau \mapsto  \lambda_1^+(\tau):=\min(\sigma(\Dirac_\tau)\cap(m,+\infty)),
\end{equation}
which assigns to every $\tau\in\R$ the first (lowest) positive eigenvalue of $\Dirac_\tau$. In \cite[Theorem 1.5]{Mas2022} it is shown that the function
$\lambda_1^+$ is continuous and increasing in $\R$, that 
$\lim_{\tau\downarrow-\infty}\lambda_1^+(\tau)=m$, 
and that 
$\lim_{\tau\uparrow+\infty}\lambda_1^+(\tau)\in(m,+\infty]$. However, in \cite{Mas2022} it is not proven that 
$\lim_{\tau\uparrow+\infty}\lambda_1^+(\tau)<+\infty$, a property that would relate this limit with the spectrum of the Dirichlet Laplacian thanks to \Cref{AMSV:eig_curv} $(ii)$.
In this regard, in \cite[Remark 4.4]{Mas2022} it is pointed out that if there was strong resolvent convergence of $\Dirac_\tau$ to $\Dirac_{+\infty}$ as $\tau \uparrow +\infty$ then the finiteness of the limit would be ensured. 
As a consequence of \Cref{thm:res_conv_1}, in the following result we show that 
$\lim_{\tau\uparrow+\infty}\lambda_1^+(\tau)<+\infty$ and that, indeed, this limit is described in terms of the first eigenvalue of the Dirichlet Laplacian. 

\begin{corollary}\label{LimitFirstEigenvalue}
Let $\lambda_1^+$ be defined by \eqref{defi:ParamFirstEig}. Then,
       $ \lim_{\tau\uparrow+\infty} \lambda_1^+(\tau) = \sqrt{\min \sigma(-\Delta_D)+m^2}.$
\end{corollary}

By the argument provided at the beginning of this section, 
$\Dirac_\tau$ does not converge in the norm resolvent sense to $\Dirac_{\pm \infty}$ as $\tau\to \pm \infty$, since $\mp m\in\sigma_{\operatorname{ess}}(\Dirac_{\pm\infty})$ (and, therefore, 
$\sigma_{\operatorname{ess}}(\Dirac_{\pm\infty})\neq\emptyset$). It is then natural to ask whether the norm resolvent convergence could be achieved if, in some sense, the study was restricted to 
$\sigma(\Dirac_{\pm\infty})\setminus\{\mp m\}$.
An affirmative answer holds true in the following sense. Denote 
\begin{equation}
\begin{split}
&\ker(\Dirac_{\pm\infty}\pm m):=\{\psi \in \operatorname{Dom}(\Dirac_{\pm\infty})\subset L^2(\Omega)^4: 
(\Dirac_{\pm\infty} \pm m)\psi=0\},\\
&\ker(\Dirac_{\pm\infty}\pm m)^{\perp}:=
\{\varphi\in L^2(\Omega)^4: 
\langle \varphi, \psi\rangle_{L^2(\Omega)^4}=0\text{ for all }
\psi\in\ker(\Dirac_{\pm\infty}\pm m)\}.
\end{split}
\end{equation}
Since $\ker(\Dirac_{\pm\infty}\pm m)^{\perp}$ is a closed subspace of $L^2(\Omega)^4$, the orthogonal projection 
\begin{equation}\label{def:ort_proj}
P_{\pm}:L^2(\Omega)^4\to \ker(\Dirac_{\pm\infty}\pm m)^{\perp}\subset L^2(\Omega)^4
\end{equation}
is a well-defined bounded self-adjoint operator in $L^2(\Omega)^4$. Moreover, from \eqref{SpectraPMIn} we know that $\ker(\Dirac_{\pm\infty}\pm m)^{\perp}\neq\{0\}$ and, thus, $\|P_{\pm}\|_{L^2(\Omega)^4\to L^2(\Omega)^4}=1$.

\begin{theorem}\label{thm:res_conv_2}
Given $\tau\in\R$, let $\Dirac_\tau$ be the operator defined in \eqref{def:Dirac_tau}. Let $\Dirac_{+ \infty}$ and 
$\Dirac_{-\infty}$ be the operators defined in \eqref{def:Dirac_inft} and \eqref{def:Dirac_-inft}, respectively. 
Then,
\begin{equation}
\lim_{\tau\to\pm\infty}
\big\|P_{\pm}\big((\Dirac_{\pm\infty}-\lambda)^{-1} - (\Dirac_\tau-\lambda)^{-1}\big)\big\|_{L^2(\Omega)^4\to L^2(\Omega)^4}=0\quad\text{for all $\lambda\in\C\setminus\R$,}
\end{equation}
where $P_\pm$ are the orthogonal projections defined in \eqref{def:ort_proj}.
\end{theorem}

As we already mentioned, the difference of resolvents 
$(\Dirac_{\pm\infty}-\lambda)^{-1} - (\Dirac_\tau-\lambda)^{-1}$
does not converge to zero in norm as $\tau \to \pm \infty$. However, if we write this difference as
\begin{equation}
\begin{split}
    (\Dirac_{\pm\infty}-\lambda)^{-1} - (\Dirac_\tau-\lambda)^{-1} &= \big(P_{\pm}+(1-P_\pm)\big)\big((\Dirac_{\pm\infty}-\lambda)^{-1} - (\Dirac_\tau-\lambda)^{-1}\big),
\end{split}
\end{equation}
then \Cref{thm:res_conv_2} shows that the eigenvalue $\mp m$ is indeed the only obstruction for having norm resolvent convergence of $\Dirac_\tau$ to $\Dirac_{\pm \infty}$ as $\tau \to \pm \infty$.

The proofs of \Cref{thm:res_conv_1,thm:res_conv_2} are given in \Cref{s:strong_res_conv,s:norm_res_conv}, respectively. In both proofs, we will only address the case $\tau\uparrow+\infty$, and 
$\varphi =(u, v)^\intercal$ will denote an element of 
$\operatorname{Dom}(\Dirac_{+\infty})$. The case 
$\tau\downarrow-\infty$ follows by analogous arguments (and will be omitted): one simply has to interchange the roles of $u$ and $v$ within the proofs. The proof of \Cref{LimitFirstEigenvalue} is given at the end of \Cref{s:strong_res_conv}. 

Concerning the approaches used in the proofs,  \Cref{thm:res_conv_1} will follow from showing that $\Dirac_{\pm \infty}$ is the strong graph limit of $\Dirac_\tau$ as $\tau\to\pm\infty$ and the equivalence of the notions of strong resolvent convergence and strong graph limit in the case of self-adjoint operators; see \cite[Theorem VIII.26]{ReedSimon1980}. Instead, for the proof of \Cref{thm:res_conv_2} we will directly estimate the pairing 
$$\big\langle P_{\pm}\big((\Dirac_{\pm\infty}-\lambda)^{-1} - (\Dirac_\tau-\lambda)^{-1}\big) f, g\big\rangle_{L^2(\Omega)^4}$$ 
in terms of $\|f\|_{L^2(\Omega)^4}$, $\|g\|_{L^2(\Omega)^4}$, and $\tau$. A key step of this approach is to show that the inclusion of 
$\mathrm{Dom}(\Dirac_{\pm\infty})\cap\ker(\Dirac_{\pm\infty}\pm m)^{\perp}$ in $H^1(\Omega)^4$ is well defined and continuous with respect to the graph norm of $\Dirac_{\pm\infty}$; see \Cref{Lemma:AlphaGradEstimates_ort}.

We mention that one can also prove \Cref{thm:res_conv_1} with the approach used to prove \Cref{thm:res_conv_2}, but now estimating
\begin{equation} \label{eq:pairing}
    \big\langle \big((\Dirac_{\pm\infty}-\lambda)^{-1} - (\Dirac_\tau-\lambda)^{-1}\big) f, g\big\rangle_{L^2(\Omega)^4}
\end{equation} 
and using a density argument that leads to the fact that 
$(\Dirac_\tau-\lambda)^{-1}$ converges weakly to $(\Dirac_{\pm\infty}-\lambda)^{-1}$ as $\tau \to \pm \infty$; see \cite[Section 3.2]{Duran2024} for the details. With this ingredient in hand, \Cref{thm:res_conv_1} follows from the equivalence of the notions of strong resolvent convergence and weak resolvent convergence in the setting of self-adjoint operators; see \cite[Lemma 6.37]{Teschl2014}, \cite[Exercise 20.$(a)$ of Chapter VIII]{ReedSimon1980}, or \cite[Theorem 2.16]{Duran2024}.

Although the main interest in the present article is the study of the convergence in a resolvent sense as $\tau \to \pm \infty$, we  mention that the study for $\tau$ tending to any finite value $\tau_0\in\R$ was carried out in \cite[Theorem 1.10]{Duran2024}, where the following result is proven.

\begin{theorem} \label{thm:res_conv_norm} 
Given $\tau \in\R$, let $\Dirac_\tau$ be the operator defined in \eqref{def:Dirac_tau}. Then, for every $\tau_0\in\R$, $\Dirac_\tau$ converges to $\Dirac_{\tau_0}$ in the norm resolvent sense as $\tau \to \tau_0$.
\end{theorem}

The proof in \cite{Duran2024} of this theorem is based on the fact that the resolvent operator 
$(\Dirac_\tau-\lambda)^{-1}$ is real analytic in $\tau$ in a neighborhood of $\tau_0$; see the proof of \cite[Lemma~3.1]{Mas2022} for this last assertion. As a final contribution of this article, in \Cref{Sec:Continuity_Resolvent} we give another proof of \Cref{thm:res_conv_norm} in the line of our proof of \Cref{thm:res_conv_2}, that is, via estimating the pairing 
\begin{equation} \label{eq:pairing_tau0}
    \big\langle \big((\Dirac_{\tau_0}-\lambda)^{-1} - (\Dirac_\tau-\lambda)^{-1}\big) f, g\big\rangle_{L^2(\Omega)^4}
\end{equation}
in terms of $\|f\|_{L^2(\Omega)^4}$, $\|g\|_{L^2(\Omega)^4}$,  $\tau$, and $\tau_0$. In addition, in \Cref{Sec:Continuity_Resolvent} we give a quantitative estimate in terms of $\tau$ of the norm of $(\Dirac_\tau-\lambda)^{-1}$ as a bounded operator from $L^2(\Omega)^4$ to $H^1(\Omega)^4$; see
\Cref{thm:ContinuityResolvent}. The latter may be of interest for future references.

\subsection{Quantum dots in $\R^2$}\label{ss:quantum_dots}
The two-dimensional analogue of the Dirac differential operator $$-i\alpha\cdot\nabla+m\beta=-i(\alpha_1\partial_1+\alpha_2\partial_2+\alpha_3\partial_3)+m\beta$$ is obtained  replacing the vectors $(\alpha_1,\alpha_2,\alpha_3)$ and $(\partial_1,\partial_2,\partial_3)$ by $(\sigma_1,\sigma_2)$ and $(\partial_1,\partial_2)$, respectively, and $\beta$ by $\sigma_3$. In this regard, for 
$\Omega\subset \R^2$ the corresponding Dirac operator 
$\Dirac_\tau$ would be
\begin{equation}\label{def:Dirac_tau_2}
	\begin{split}
		\mathrm{Dom}(\Dirac_\tau) &:= \big\{ \varphi \in H^1(\Omega)^2: \, \varphi =i (\sinh\tau- \cosh\tau \, \sigma_3)( \sigma_1\nu_1+\sigma_2 \nu_2 ) \varphi  \,\text{ on } \partial \Omega \big\},\\
		\Dirac_\tau\varphi &:= (-i(\sigma_1\partial_1+\sigma_2\partial_2)+m\sigma_3)\varphi \quad\text{for all 
			$\varphi\in\mathrm{Dom}(\Dirac_\tau),$}
	\end{split}
\end{equation}
and analogously for the two-dimensional version of 
$\Dirac_{\pm\infty}$.

Motivated by their applications in the description of graphene quantum dots and nano-ribbons, in \cite{Benguria2017Self,Benguria2017Spectral} it was studied the following family of two-dimensional Dirac operators, which we shall see afterward that is closely related to the family $\{\Dirac_\tau\}_{\tau\in\R\cup\{\pm\infty\}}$.
Given $\Omega\subset \R^2$ and $\eta\in(-\pi,\pi]\setminus\{-\pi/2,\pi/2\}$, let $\D_\eta$ be the operator in $L^2(\Omega)^2$ defined by
\begin{equation}\label{def:Dirac_eta_2}
	\begin{split}
		\mathrm{Dom}(\D_\eta) &:= \big\{ \varphi \in H^1(\Omega)^2: \, \varphi =(\cos\eta\,(\sigma_1 t_1+\sigma_2 t_2)+\sin\eta\, \sigma_3) \varphi  \,\text{ on } \partial \Omega \big\},\\
		\D_\eta\varphi &:= (-i(\sigma_1\partial_1+\sigma_2\partial_2)+m\sigma_3)\varphi \quad\text{for all 
			$\varphi\in\mathrm{Dom}(\D_\eta),$}
	\end{split}
\end{equation}
where $t$ denotes the unit vector tangent to
$\partial\Omega$ with the orientation of $t$ chosen such that $\{\nu,t\}$ is positively oriented. The so-called infinite mass boundary conditions correspond to $\eta\in\{0, \pi\}$. The  zigzag boundary conditions formally correspond to 
$\eta\in\{-\pi/2,\pi/2\}$, and give rise to the operators 
$\D_{\pm\pi/2}$ defined by
\begin{equation}
\begin{split}
\mathrm{Dom}(\D_{-\pi/2}) 
&:= \big\{ \varphi= (u, v)^\intercal : \, u \in H^1_0(\Omega),  v \in L^2(\Omega), (\sigma_1\partial_1+\sigma_2\partial_2) \varphi \in L^2(\Omega)^2 \big\},\\
\D_{-\pi/2}\varphi &:= (-i(\sigma_1\partial_1+\sigma_2\partial_2) + m \sigma_3)\varphi \quad\text{for all $\varphi\in\mathrm{Dom}(\D_{-\pi/2})$}
\end{split}
\end{equation}
and
\begin{equation}
\begin{split}
\mathrm{Dom}(\D_{\pi/2}) 
&:= \big\{ \varphi= (u, v)^\intercal : \, u \in L^2(\Omega), v \in H^1_0(\Omega), (\sigma_1\partial_1+\sigma_2\partial_2) \varphi \in L^2(\Omega)^2 \big\},\\
\D_{\pi/2}\varphi &:= (-i(\sigma_1\partial_1+\sigma_2\partial_2) + m \sigma_3)\varphi \quad\text{for all $\varphi\in\mathrm{Dom}(\D_{\pi/2}).$}
\end{split}
\end{equation}

The purpose of this section is to clarify the relation between the families $\{\Dirac_\tau\}_{\tau\in\R\cup\{\pm\infty\}}$ (described in \eqref{def:Dirac_tau_2}) and $\{\D_\eta\}_{\eta\in(-\pi,\pi]}$, and to see how the results of the present paper apply to the setting of graphene quantum dots. 

A simple computation shows that 
$\sigma_1 t_1+\sigma_2 t_2=i( \sigma_1\nu_1+\sigma_2 \nu_2 )\sigma_3$.
Using this identity together with the algebraic properties of the matrices $\sigma_1$, $\sigma_2$, and $\sigma_3$ it is straightforward to check that the boundary condition $\varphi =i (\sinh\tau- \cosh\tau \, \sigma_3)( \sigma_1\nu_1+\sigma_2 \nu_2 ) \varphi$ for $\Dirac_\tau$ is equivalent to 
$$\varphi =\Big(\frac{1}{\cosh\tau}(\sigma_1 t_1+\sigma_2 t_2)- \frac{\sinh\tau}{\cosh\tau} \, \sigma_3\Big) \varphi.$$
This can be rewritten as $\varphi =(\cos\eta\,(\sigma_1 t_1+\sigma_2 t_2)+\sin\eta\, \sigma_3) \varphi$ if we take $\eta$ such that
\begin{equation}\label{genMIT_quantdot}
\cos\eta=\frac{1}{\cosh\tau} \quad\text{and}\quad 
\sin\eta=-\frac{\sinh\tau}{\cosh\tau}.
\end{equation}
Indeed, since $(1/\cosh\tau)^2+(-\sinh\tau/\cosh\tau)^2=1$ and $0<1/\cosh\tau\leq1$, for every 
$\tau\in\R$ there exists a unique $\eta\in(-\pi/2,\pi/2)$ such that 
\eqref{genMIT_quantdot} holds, and vice versa. This shows that 
$\{\Dirac_{\tau}\}_{\tau\in\R}=\{\D_{\eta}\}_{\eta\in(-\pi/2,\pi/2)}$ via the correspondence \eqref{genMIT_quantdot}.\footnote{This correspondence between (the two-dimensional version of) the generalized MIT bag model $\Dirac_\tau$ and the quantum dot boundary condition in $\D_\eta$ was first pointed out to us by B. Cassano in July 2022, and mentioned later on in the introduction of \cite{Behrndt2023}.} Similarly, 
if we define 
\begin{equation}
	\begin{split}
		\mathrm{Dom}(\mathcal{T}_\tau) &:= \big\{ \varphi \in H^1(\Omega)^2: \, \varphi =i (\sinh\tau+ \cosh\tau \, \sigma_3)( \sigma_1\nu_1+\sigma_2 \nu_2 ) \varphi  \,\text{ on } \partial \Omega \big\},\\
		\mathcal{T}_\tau\varphi &:= (-i(\sigma_1\partial_1+\sigma_2\partial_2)+m\sigma_3)\varphi \quad\text{for all 
			$\varphi\in\mathrm{Dom}(\mathcal{T}_\tau)$}
	\end{split}
\end{equation}
for $\tau\in\R$ (note that the definition of $\mathcal{T}_\tau$ is the same as the one of $\Dirac_\tau$ in \eqref{def:Dirac_tau_2} except for replacing 
$\cosh\tau$ by $-\cosh\tau$ in the boundary condition), then
$\{\mathcal{T}_{\tau}\}_{\tau\in\R}=\{\D_\eta\}_{\eta\in(-\pi,-\pi/2)\cup(\pi/2,\pi]}$ via the correspondence
\begin{equation}
\cos\eta=-\frac{1}{\cosh\tau} \quad\text{and}\quad 
\sin\eta=\frac{\sinh\tau}{\cosh\tau}.
\end{equation}
Finally, in what regards the boundary condition, the cases 
$\eta=\pm \pi/2$ for $\D_\eta$ correspond to $\tau=\mp\infty$ for $\Dirac_\tau$ and to $\tau=\pm\infty$ for $\mathcal{T}_\tau$, respectively. 

Our proofs of \Cref{thm:res_conv_1,thm:res_conv_2} adapt with no difficulties to the two-dimensional version of 
$\Dirac_\tau$ given in \eqref{def:Dirac_tau_2}, and also to 
$\mathcal{T}_\tau$.
Hence, taking into account the previous observations, we get that \Cref{thm:res_conv_1,thm:res_conv_2} hold true replacing 
$\Dirac_\tau$ by $\D_\eta$, 
$\Dirac_{\pm\infty}$ by $\D_{\pm\pi/2}$, $\tau\to\pm\infty$ by $\eta\to\pm\pi/2$, 
and $\ker(\Dirac_{\pm\infty}\pm m)^{\perp}$ by
$\ker(\D_{\pm\pi/2}\mp m)^{\perp}$. In the same way, the analogue of \Cref{thm:res_conv_norm} for $\D_\eta$ with 
$\eta\to\eta_0\in(-\pi,\pi]\setminus\{-\pi/2,\pi/2\}$ holds true.

\subsection{Notation}\label{ss:notation}
In this section we recall some basic notation regarding the Hilbert spaces and associated norms to be used throughout the paper. 

In the sequel, $\Omega$ denotes a bounded domain in 
$\R^3$ with $C^2$ boundary. Let $d\geq1$ and $k\geq1$ be integers. We denote by $L^2(\Omega)^d$ the Hilbert space of  functions 
$\varphi:\Omega\to\C^d$ endowed with the scalar product and the associated norm
\begin{equation}
\langle \varphi,\psi\rangle_{L^2(\Omega)^d}:=\int_{\Omega} \varphi\cdot\overline\psi \,dx \quad\text{and}\quad
\|\varphi\|_{L^2(\Omega)^d}:=\langle \varphi,\varphi\rangle_{L^2(\Omega)^d}^{1/2},
\end{equation}
respectively. We denote by
$H^k(\Omega)^d$ the Sobolev space of 
functions in $L^2(\Omega)^d$ with weak partial derivatives up to order $k$ in $L^2(\Omega)^d$, and $H_0^k(\Omega)^d$ denotes the closure with respect to the $H^k(\Omega)^d$-norm of the set of smooth functions compactly supported in 
$\Omega$. 

Similarly, $L^2(\partial\Omega)^d$ denotes the Hilbert space of  functions $\varphi:\partial\Omega\to\C^d$ endowed with the scalar product and the associated norm
\begin{equation}
\langle \varphi,\psi\rangle_{L^2(\partial\Omega)^d}:=\int_{\partial\Omega} \varphi\cdot\overline\psi \,d\upsigma \quad\text{and}\quad
\|\varphi\|_{L^2(\partial\Omega)^d}:=\langle \varphi,\varphi\rangle_{L^2(\partial\Omega)^d}^{1/2},
\end{equation}
respectively, where $\upsigma$ denotes the surface measure on $\partial\Omega$.
For $s\in(0,1)$, we denote by $H^s(\partial \Omega)^d$ the fractional Sobolev space of functions $\varphi\in L^2(\partial\Omega)^d$ such that 
    $$
        \|\varphi\|_{H^s(\partial \Omega)^d}:= \Big( \int_{\partial\Omega} |\varphi|^2\,d\upsigma + \int_{\partial\Omega}\int_{\partial\Omega}\frac{|\varphi(x)-\varphi(y)|^2}{|x-y|^{2+2s}} \,d\upsigma(y)\,d\upsigma(x) \Big)^{1/2} < +\infty.
    $$
    The continuous dual of $H^s(\partial \Omega)^d$ is denoted by $H^{-s}(\partial \Omega)^d$. The action of $\psi\in H^{-s}(\partial \Omega)^d$ on $\varphi \in H^s(\partial \Omega)^d$ is denoted by $\langle \psi, \varphi \rangle_{H^{-s}(\partial\Omega)^d, H^s(\partial\Omega)^d}$, and the norm in $H^{-s}(\partial \Omega)^d$ is
    $$
        \|\psi\|_{H^{-s}(\partial \Omega)^d}:= 
        { \sup_{\|\varphi\|_{H^s(\partial \Omega)^d}\leq1}} 
        \langle \psi, \varphi \rangle_{H^{-s}(\partial\Omega)^d, H^s(\partial\Omega)^d} .
    $$
    Recall that
\begin{equation}\label{eq:dual_pairing_L2}
\langle \psi, \varphi \rangle_{H^{-s}(\partial\Omega)^d, H^s(\partial\Omega)^d}=\langle \varphi, \psi \rangle _{L^2(\partial\Omega)^d}
\end{equation}
    whenever $\psi\in L^2(\partial\Omega)^d\subset H^{-s}(\partial\Omega)^d$ and $\varphi\in H^{s}(\partial\Omega)^d\subset L^2(\partial\Omega)^d$; see for example \cite[Remark~3 in Section 5.2]{Brezis2011}. The reason why the functions $\varphi$ and $\psi$ do not appear in the same order in both sides of \eqref{eq:dual_pairing_L2} is that we defined 
$\langle \cdot, \cdot \rangle _{L^2(\partial\Omega)^d}$ to be linear with respect to the first entry.

\section{Convergence in the strong resolvent sense}
\label{s:strong_res_conv}

In this section we will prove \Cref{thm:res_conv_1} and \Cref{LimitFirstEigenvalue}.
To prove \Cref{thm:res_conv_1} we will show that 
$\Dirac_{+\infty}$ is the strong graph limit of $\Dirac_\tau$ as 
$\tau\uparrow+\infty$, and we will do it in a completely constructive way, that is, exhibiting the functions in 
$\mathrm{Dom}(\Dirac_\tau)$ which converge in the strong graph sense to a given function in $\mathrm{Dom}(\Dirac_{+\infty})$. Within the proof we will use the following density result from \cite[Proposition 2.12]{Vega2016}, whose proof is nonconstructive since it is based in orthogonality arguments. However, the reader can find a constructive proof of this density result in \cite[Corollary 2.28]{Duran2024}.

\begin{lemma}\label{l:density}
Let $v\in L^2(\Omega)^2$ be such that $\sigma\cdot\nabla v \in L^2(\Omega)^2$. Then, for every 
$\epsilon>0$ there exists $v_\epsilon\in H^1(\Omega)^2$ such that
\begin{equation}
\|v-v_\epsilon\|_{L^2(\Omega)^2}+
\|\sigma\cdot\nabla v- \sigma\cdot\nabla v_\epsilon\|_{L^2(\Omega)^2}<\epsilon.
\end{equation}
\end{lemma}

With this ingredient in hand, we can now address the proof of strong resolvent convergence, as $\tau \uparrow +\infty$, of the operators $\Dirac_\tau$ to $\Dirac_{+\infty}$ via showing that $\Dirac_{+\infty}$ is the strong graph limit of $\Dirac_\tau$.

\begin{proof}[Proof of \Cref{thm:res_conv_1}]

We will prove that, given $\varphi \in \operatorname{Dom}(\Dirac_{+\infty})$, for every $\epsilon>0$ there exists 
$\tau_0\in\R$ such that for every $\tau>\tau_0$ there exists $\varphi_\tau \in \operatorname{Dom}(\Dirac_\tau)$ such that 
\begin{equation}\label{claim:res_conv}
\|\varphi-\varphi_\tau\|_{L^2(\Omega)^4}
+\|(-i\alpha\cdot\nabla+m\beta)(\varphi-\varphi_\tau)\|_{L^2(\Omega)^4}<\epsilon.
\end{equation}
Once this is proven, it is straightforward to check that $\Dirac_{+\infty}$ is the strong graph limit of $\Dirac_\tau$ as $\tau\uparrow+\infty$, and then the convergence of $\Dirac_\tau$ to $\Dirac_{+\infty}$ in the strong resolvent sense as $\tau\uparrow+\infty$ follows by \cite[Theorem VIII.26]{ReedSimon1980}.

Let $\varphi= (u, v)^\intercal \in 
\operatorname{Dom}(\Dirac_{+\infty})$ and $\delta>0$ small enough to be chosen later on. By \Cref{l:density}, there exists 
$v_\delta\in H^1(\Omega)^2$ such that
\begin{equation}\label{eq:density_delta}
\|v-v_\delta\|_{L^2(\Omega)^2}+
\|\sigma\cdot\nabla (v-v_\delta)\|_{L^2(\Omega)^2}<\delta.
\end{equation} 
Since $v_\delta\in H^1(\Omega)^2$, by the trace theorem we have $v_\delta\in H^{1/2}(\partial\Omega)^2$. Next, for every $\tau\in\R$ we set 
$v_{\delta,\tau} := -ie^{-\tau}(\sigma\cdot\nu)v_\delta$  on $\partial \Omega$.
Then, using that $\nu$ is of class $C^1$ on $\partial\Omega$ and  \cite[Lemma A.2]{Behrndt2020}, we deduce that 
\begin{equation}
\|v_{\delta,\tau}\|_{H^{1/2}(\partial\Omega)^2}\leq Ce^{-\tau}
\|v_\delta\|_{H^{1/2}(\partial\Omega)^2}
\end{equation} 
for some $C>0$ only depending on $\Omega$. Let  
$E: H^{1/2}(\partial \Omega)^2 \rightarrow H^1(\Omega)^2$ be a bounded linear extension operator, and set
$u_{\delta,\tau} := E(v_{\delta,\tau}) \in H^1(\Omega)^2$. 
Then, 
\begin{equation}\label{eq:bdy_cond_tau}
\|u_{\delta,\tau}\|_{H^1(\Omega)^2}\leq Ce^{-\tau}\|v_\delta\|_{H^{1/2}(\partial\Omega)^2}
\end{equation}
for some $C>0$ only depending on $\Omega$.

Finally, set $\varphi_{\delta,\tau} 
:=(u+u_{\delta,\tau},v_\delta)^\intercal$.
We know that $u_{\delta,\tau}$ and $v_\delta$ belong to $H^1(\Omega)^2$, and recall that $u\in H^1_0(\Omega)^2$ since $\varphi= (u, v)^\intercal \in \operatorname{Dom}(\Dirac_{+\infty})$. Thus, $\varphi_{\delta,\tau}\in H^1(\Omega)^4$ and, on $\partial\Omega$, we have
\begin{equation}
u+u_{\delta,\tau}=u_{\delta,\tau}
=v_{\delta,\tau}=-ie^{-\tau}(\sigma\cdot\nu)v_\delta.
\end{equation}
This last equation means that $\varphi_{\delta,\tau} = i (\sinh\tau- \cosh\tau\,  \beta)( \alpha \cdot\nu ) \varphi_{\delta,\tau}$ on $\partial\Omega$, and this leads to 
$\varphi_{\delta,\tau}\in \operatorname{Dom}(\Dirac_\tau)$. Moreover,
\begin{equation}\label{eq:key_ineq_res_conv}
\begin{split}
\|\varphi-\varphi_{\delta,\tau}\|_{L^2(\Omega)^4}
&+\|(-i\alpha\cdot\nabla+m\beta)(\varphi-\varphi_{\delta,\tau})\|_{L^2(\Omega)^4}\\
&\leq
(1+m)\|\varphi-\varphi_{\delta,\tau}\|_{L^2(\Omega)^4}
+\|\alpha\cdot\nabla(\varphi-\varphi_{\delta,\tau})\|_{L^2(\Omega)^4}\\
&=(1+m)\big(\|u_{\delta,\tau}\|^2_{L^2(\Omega)^2}
+\|v-v_{\delta}\|^2_{L^2(\Omega)^2}\big)^{1/2}\\
&\quad+\big(\|\sigma\cdot\nabla u_{\delta,\tau}\|^2_{L^2(\Omega)^2}
+\|\sigma\cdot\nabla 
(v-v_{\delta})\|^2_{L^2(\Omega)^2}\big)^{1/2}\\
&\leq\big((1+m)^2Ce^{-2\tau}\|v_\delta\|^2_{H^{1/2}(\partial\Omega)^2}
+(1+m)^2\|v-v_{\delta}\|^2_{L^2(\Omega)^2}\big)^{1/2}\\
&\quad+\big(Ce^{-2\tau}\|v_\delta\|^2_{H^{1/2}(\partial\Omega)^2}
+\|\sigma\cdot\nabla (v-v_{\delta})\|^2_{L^2(\Omega)^2}\big)^{1/2}
\end{split}
\end{equation}
for some $C>0$ depending only on $\Omega$, where we used \eqref{eq:bdy_cond_tau} in the last inequality above. Therefore, given $\epsilon>0$, using \eqref{eq:density_delta} we first take $\delta>0$ small enough such that 
\begin{equation}
(1+m)\|v-v_{\delta}\|_{L^2(\Omega)^2}<\frac{\epsilon}{4}
\quad\text{and}\quad
\|\sigma\cdot\nabla (v-v_{\delta})\|_{L^2(\Omega)^2}
<\frac{\epsilon}{4},
\end{equation}
and then, once $\delta$ is chosen, we take $\tau_0\in\R$ big enough such that 
\begin{equation}
(1+m)C^{1/2}e^{-\tau}\|v_\delta\|_{H^{1/2}(\partial\Omega)^2}<\frac{\epsilon}{4}\quad\text{for all $\tau>\tau_0$.}
\end{equation}
Plugging these estimates in \eqref{eq:key_ineq_res_conv} we conclude that
\eqref{claim:res_conv} holds taking 
$\varphi_\tau:=\varphi_{\delta,\tau}$.
\end{proof}

As a consequence of the strong resolvent convergence of $\Dirac_\tau$ to $\Dirac_{+\infty}$ as $\tau \uparrow +\infty$, we show now the relation of the limit, as $\tau \uparrow +\infty$, of the first positive eigenvalue of $\Dirac_\tau$ with the first eigenvalue of the Dirichlet Laplacian.

\begin{proof}[Proof of \Cref{LimitFirstEigenvalue}]
    Since $\Dirac_\tau$ converges to $\Dirac_{+\infty}$ in the strong resolvent sense as $\tau\uparrow+\infty$ by \Cref{thm:res_conv_1}, using 
    \cite[Theorem VIII.24 $(a)$]{ReedSimon1980} and \eqref{SpectraPMIn}, there exist $\lambda_\tau \in \sigma(\Dirac_\tau)$ such that
    \begin{equation}
       \sqrt{\min \sigma(-\Delta_D)+m^2} = \lim_{\tau \uparrow +\infty} \lambda_\tau.
    \end{equation}
In particular, $\lambda_\tau>m$ for all $\tau$ big enough, which yields $\lambda_1^+(\tau)\leq\lambda_\tau$ for all 
$\tau$ big enough by the definition of $\lambda_1^+$. Therefore, using \cite[Theorem 1.5]{Mas2022} we deduce that
     \begin{equation}
        \sqrt{\min \sigma(-\Delta_D)+m^2}\leq\lim_{\tau \uparrow +\infty} \lambda_1^+(\tau) \leq \lim_{\tau \uparrow +\infty} \lambda_\tau = \sqrt{\min \sigma(-\Delta_D)+m^2},
     \end{equation}
     and the corollary follows.
\end{proof}

\section{Convergence of the projected resolvents in operator norm}
\label{s:norm_res_conv}

This section is devoted to the proof of \Cref{thm:res_conv_2}. We begin with the following key lemma, which unveils why convergence of the resolvents in operator norm as $\tau\uparrow+\infty$ can be achieved when we restrict to $\ker(\Dirac_{+\infty}+ m)^{\perp}$. Indeed, recall that $\mathrm{Dom}(\Dirac_{+\infty})$ is not contained in $H^1(\Omega)^4$ due to the eigenvalue $-m$ of infinite multiplicity. However, as the following result shows, the inclusion of $\mathrm{Dom}(\Dirac_{+\infty})\cap\ker(\Dirac_{+\infty}+ m)^{\perp}$ in $H^1(\Omega)^4$ is well defined and continuous with respect to the graph norm of $\Dirac_{+\infty}$.

\begin{lemma}\label{Lemma:AlphaGradEstimates_ort}
Let $\varphi \in\operatorname{Dom}(\Dirac_{+\infty})$ and assume that 
$\langle \varphi, \psi\rangle_{L^2(\Omega)^4}=0$ for all 
$\psi \in \operatorname{Dom}(\Dirac_{+\infty})$ such that 
$\Dirac_{+\infty} \psi=-m\psi$.
Then, $\varphi$ belongs to $H^{1}(\Omega)^4$ and the estimate
\begin{equation}\label{ineq:cont_graph_h1} 
\| \varphi \|_{H^{1}(\Omega)^4} \leq C (   \| \varphi \|_{L^2(\Omega)^4}    + \| \alpha \cdot \nabla \varphi \|_{L^2(\Omega)^4} )
\end{equation}
holds for some $C>0$ depending only on $\Omega$.
\end{lemma}

This lemma will follow from the following result applied to the lower component of $\varphi$.

\begin{lemma}\label{Lemma:SigmaGradEstimates_ort}
Let $v \in L^2(\Omega)^2$ be such that $ \sigma \cdot \nabla v \in L^2(\Omega)^2$. Assume that 
$\langle v, w\rangle_{L^2(\Omega)^2}=0$ for all 
$w\in L^2(\Omega)^2$ such that $ \sigma \cdot \nabla w=0$ in  $\Omega$.
Then, $v$ belongs to $H^{1}(\Omega)^2$ and satisfies 
$$\|v \|_{H^{1}(\Omega)^2} \leq C (   \| v \|_{L^2(\Omega)^2}    + \| \sigma \cdot \nabla v \|_{L^2(\Omega)^2} )$$
for some $C>0$ depending only on $\Omega$.
\end{lemma}

Assuming for the moment \Cref{Lemma:SigmaGradEstimates_ort}, let us prove \Cref{Lemma:AlphaGradEstimates_ort}.

\begin{proof}[Proof of \Cref{Lemma:AlphaGradEstimates_ort}]
Writing $\varphi\in\operatorname{Dom}(\Dirac_{+\infty})$ in upper and lower components $\varphi=(u,v)^\intercal$, we see that both $v$ and $\sigma\cdot\nabla v$ belong to $L^2(\Omega)^2$, and that $u\in H_0^1(\Omega)^2$. Since the trace of $u$ on $\partial\Omega$ vanishes, by \cite[Lemma 2.1]{Mas2022} we have 
\begin{equation}\label{l:est_up_comp}
\|u \|_{H^{1}(\Omega)^2} \leq C (   \| u \|_{L^2(\Omega)^2}    + \| \sigma \cdot \nabla u \|_{L^2(\Omega)^2} )
\end{equation}
for some $C>0$ depending only on $\Omega$.

Next, given $w\in L^2(\Omega)^2$ such that $ \sigma \cdot \nabla w=0$, set $\psi_w:=(0,w)^\intercal\in\operatorname{Dom}(\Dirac_{+\infty})$. Note that
\begin{equation}
\Dirac_{+\infty}\psi_w=
\begin{pmatrix}
m & -i\sigma\cdot\nabla\\
-i\sigma\cdot\nabla & -m
\end{pmatrix}
\begin{pmatrix}
0\\
w
\end{pmatrix}
=\begin{pmatrix}
-i\sigma\cdot\nabla w\\
-mw
\end{pmatrix}
=\begin{pmatrix}
0\\
-mw
\end{pmatrix}
=-m\psi_w.
\end{equation}
Therefore, by the assumption on $\varphi$, we have
$0=\langle \varphi, \psi_w\rangle_{L^2(\Omega)^4}
=\langle v, w\rangle_{L^2(\Omega)^2}$
for all $w\in L^2(\Omega)^2$ such that $ \sigma \cdot \nabla w=0$.  Then, the lemma follows from \Cref{Lemma:SigmaGradEstimates_ort} applied to $v$ and \eqref{l:est_up_comp}. 
\end{proof}

Let us now address the proof of \Cref{Lemma:SigmaGradEstimates_ort}.

\begin{proof}[Proof of \Cref{Lemma:SigmaGradEstimates_ort}]
Set 
\begin{equation}
\Gamma(x):=\frac{1}{4\pi|x|}
\quad\text{and}\quad
K(x):=i\sigma\cdot\frac{x}{4\pi|x|^3}\quad\text{for all $x\in\R^3\setminus\{0\}$.}
\end{equation}
Observe that $K=-i\sigma\cdot\nabla \Gamma$ and that
$-\Delta\Gamma=\delta_0$ in the sense of distributions, where $\delta_0$ denotes the Dirac delta measure in $\R^3$ centered at the origin of coordinates. This, together with the fact that 
$(-i\sigma\cdot\nabla)^2=-\Delta$, leads to
$-i\sigma\cdot\nabla K=\delta_0$ 
in the sense of distributions. 

Let $v \in L^2(\Omega)^2$, and define 
\begin{equation} \label{def:Newt_pot}
v_\Gamma(x):=\int_\Omega \Gamma(x-y)v(y)\,dy
\quad\text{and}\quad
v_K(x):=\int_\Omega K(x-y)v(y)\,dy
\quad\text{for a.e. $x\in\Omega$}.
\end{equation}
Since $v_\Gamma$ is the Newtonian potential of $v \in L^2(\Omega)^2$, by \cite[Lemma 7.12 and Theorem 9.9]{Gilbarg-Trudinger} we get that $v_\Gamma\in H^2(\Omega)^2$ with $\|v_\Gamma\|_{L^2(\Omega)^2} \leq C \|v\|_{L^2(\Omega)^2}$ and $\|D^2 v_\Gamma\|_{L^2(\Omega)^2} \leq C \|v\|_{L^2(\Omega)^2}$ for some  $C>0$ depending only on $\Omega$. Using now the Gagliardo-Nirenberg interpolation inequality in bounded domains \cite[Theorem 1]{Nirenberg1966}, we indeed have that $\|v_\Gamma\|_{H^2(\Omega)^2} \leq C \|v\|_{L^2(\Omega)^2}$ for some $C>0$ depending only on 
$\Omega$. In particular, since $K=-i\sigma\cdot\nabla \Gamma$, it holds that  $v_K=-i\sigma\cdot\nabla v_\Gamma$ and, thus, $v_K\in H^1(\Omega)^2$ with 
\begin{equation}\label{eq:vk_H1_v_L2}
\|v_K\|_{H^1(\Omega)^2} \leq C \|v\|_{L^2(\Omega)^2}
\end{equation} 
for some $C>0$ depending only on $\Omega$ ---this estimate also follows from \cite[Proposition 2.17]{Vega2016}. 

Given $g\in C^\infty(\partial\Omega)^2$ set
\begin{equation}
w_g(x):=\int_{\partial\Omega} K(x-y)g(y)\,d\upsigma(y)
\quad\text{for all $x\in\Omega$}.
\end{equation}
It is well known that $w_g\in L^2(\Omega)^2$. For the sake of completeness, we give here a simple proof of this assertion following the ideas from \cite[Lemma 2.1]{AMV2014}. By the Cauchy-Schwarz inequality,
\begin{equation}
|w_g(x)|^2\leq\int_{\partial\Omega} |K(x-y)|^{3/4}\,d\upsigma(y)\int_{\partial\Omega} |K(x-z)|^{5/4}|g(z)|^{2}\,d\upsigma(z)
\quad\text{for all $x\in\Omega$}.
\end{equation}
Note that 
$|K(x)|^{3/4}\leq C|x|^{-3/2}$ for all $x\in\R^3\setminus\{0\}$ ---by $|K(x)|$ we mean the Frobenius norm of the matrix $K(x)$, that is, the square root of the sum of the squares of its entries. Using this, that $-3/2>-2$, and that $\partial\Omega$ is a bounded $C^2$ surface, it is an exercise to show that
\begin{equation}
\sup_{x\in\Omega}\int_{\partial\Omega} |K(x-y)|^{3/4}\,d\upsigma(y)<+\infty.
\end{equation}
Therefore, since $|K(x)|^{5/4}\leq C|x|^{-5/2}$ for all $x\in\R^3\setminus\{0\}$, using Fubini's theorem, that 
$-5/2>-3$, and that $\Omega$ is bounded, we obtain
\begin{equation}
\|w_g\|_{L^2(\Omega)^2}^2\leq C\int_{\partial\Omega}|g(z)|^{2}\int_\Omega |K(x-z)|^{5/4}\,dx\,d\upsigma(z)
\leq C\|g\|_{L^2(\partial\Omega)^2}^2
\end{equation}
for some $C>0$ depending only on $\Omega$. This proves that $w_g\in L^2(\Omega)^2$. Moreover, $w_g$ is continuously differentiable in $\Omega$. In particular, using that $-i\sigma\cdot\nabla K(x)=0$ for all $x\in\R^3\setminus\{0\}$, we get that $\sigma \cdot \nabla w_g(x)=0$ for all $x\in\Omega$.
 
Assume now that $v \in L^2(\Omega)^2$ satisfies $\langle v, w\rangle_{L^2(\Omega)^2}=0$ 
for all $w\in L^2(\Omega)^2$ such that $ \sigma \cdot \nabla w=0$ in $\Omega$. Then, for every $g\in C^\infty(\partial\Omega)^2$, we have
\begin{equation}\label{eq:norm_conv_zero_trace}
\begin{split}
0=\langle v, w_g\rangle_{L^2(\Omega)^2}
&=\int_\Omega \int_{\partial\Omega} v(x)\cdot\overline{K(x-y)g(y)}\,d\upsigma(y)\,dx\\
&=\int_{\partial\Omega} \int_\Omega K(y-x)v(x)\cdot\overline{g(y)}\,dx\,d\upsigma(y)
=\langle v_K, g\rangle_{L^2(\partial\Omega)^2}.
\end{split}
\end{equation}
A comment on the third and fourth equalities in \eqref{eq:norm_conv_zero_trace} are in order.\footnote{The equality 
$\langle v, w_g\rangle_{L^2(\Omega)^2}
=\langle v_K, g\rangle_{L^2(\partial\Omega)^2}$ can also be proven using that $v=-i\sigma\cdot\nabla v_K$ and integration by parts. However, this approach would require to introduce the Plemelj-Sokhotski jump formulas for $w_g$, which are not used in the rest of this article.} To justify the third equality we used Fubini's theorem and the estimate
\begin{equation}\label{eq:dom_conv_estimate}
\begin{split}
\Big(\int_\Omega &\int_{\partial\Omega} |v(x)||K(x-y)||g(y)|\,d\upsigma(y)\,dx\Big)^{2}\\
&\leq\Big(\int_\Omega |v(x)|^2 \int_{\partial\Omega} |K(x-y)|^{3/4}\,d\upsigma(y)\,dx\Big)
\Big(\int_{\partial\Omega}|g(\tilde y)|^2\int_\Omega  |K(\tilde x-\tilde y)|^{5/4}\,d\tilde x\,d\upsigma(\tilde y)\Big)\\
&<+\infty,
\end{split}
\end{equation}
which follows similarly to what we argued to prove that $w_g\in L^2(\Omega)^2$. 

Next, we justify the fourth equality in \eqref{eq:norm_conv_zero_trace}, in which the term $v_K$ on the right hand side denotes the trace on $\partial\Omega$ of $v_K\in H^1(\Omega)^2$. To do it, we will follow the ideas from the proof of \cite[Lemma 2.10]{AMV2014}. Given 
$\epsilon>0$ set $\Omega_\epsilon := \{x\in \Omega: \operatorname{dist}(x, \partial \Omega)>\epsilon \}$ and 
$v_\epsilon:= \chi_{\Omega_\epsilon} v$, where $\chi_{\Omega_\epsilon}$ is the characteristic function of $\Omega_\epsilon$. Let $(v_\epsilon)_K$ be defined as in $\eqref{def:Newt_pot}$ replacing $v$ by $v_\epsilon$. 
On the one hand, note that $(v_\epsilon)_K$ is continuous in a neighborhood of $\partial\Omega$ since $v_\epsilon$ vanishes outside $\Omega_\epsilon$. Hence,
the trace of $(v_\epsilon)_K$ on 
$\partial\Omega$  is given by the formula
\begin{equation}\label{eq:trace_approx_vK_vKeps2}
(v_\epsilon)_K(y) = \int_{\Omega} K(y-x)v_\epsilon(x) \,dx = \int_{\Omega_\epsilon} K(y-x)v(x) \,dx \quad \text{for all } y\in \partial \Omega.
\end{equation}
On the other hand, by the trace theorem and \eqref{eq:vk_H1_v_L2} applied to $v_\epsilon-v$, we have
\begin{equation}\label{eq:trace_approx_vK_vKeps}
 \|(v_\epsilon)_K - v_K\|_{L^2(\partial\Omega)^2} \leq C \|(v_\epsilon)_K - v_K\|_{H^1(\Omega)^2} \leq C \|v_\epsilon - v\|_{L^2(\Omega)^2},
\end{equation}
for some $C>0$ depending only on $\Omega$. 
Now, applying \eqref{eq:trace_approx_vK_vKeps} and \eqref{eq:trace_approx_vK_vKeps2}, and then using dominated convergence thanks to \eqref{eq:dom_conv_estimate}, we conclude that
\begin{equation}
    \begin{split}
        \langle v_K, g\rangle_{L^2(\partial\Omega)^2} & = 
        \lim_{\epsilon \downarrow 0}  \,\langle (v_\epsilon)_K, g\rangle_{L^2(\partial\Omega)^2} 
        = \lim_{\epsilon \downarrow 0}  \int_{\partial \Omega} \int_{\Omega_\epsilon} K(y-x)v(x) \cdot \overline{g(y)} \,dx\,d\upsigma(y) \\
        & = \int_{\partial \Omega} \int_{\Omega} K(y-x)v(x) \cdot \overline{g(y)} \,dx\,d\upsigma(y).
    \end{split}
\end{equation}

Once \eqref{eq:norm_conv_zero_trace} is fully justified, since $v_K\in H^1(\Omega)^2$ and \eqref{eq:norm_conv_zero_trace} holds for all $g\in C^\infty(\partial\Omega)^2$, we deduce that the trace of $v_K$ on $\partial\Omega$ vanishes and, therefore,
$v_K\in H^1_0(\Omega)^2$.

Finally, assume as before that $v \in L^2(\Omega)^2$ satisfies $\langle v, w\rangle_{L^2(\Omega)^2}=0$ 
for all $w\in L^2(\Omega)^2$ such that $ \sigma \cdot \nabla w=0$ in $\Omega$ and now, in addition, assume that $\sigma \cdot \nabla v \in L^2(\Omega)^2$. Recall that $-i\sigma\cdot\nabla K=\delta_0$, which leads to $-i\sigma\cdot\nabla v_K=v$ by the definition of $v_K$ in terms of $v$. Therefore, in the sense of distributions in $\Omega$, it holds that
\begin{equation}
-\Delta v_K=(-i\sigma\cdot\nabla)^2 v_K
=-i\sigma\cdot\nabla v \in L^2(\Omega)^2.
\end{equation}
Since $v_K\in H^1_0(\Omega)^2$, we get that
$-\Delta v_K\in L^2(\Omega)^2$ in the weak sense in $H^1_0(\Omega)^2$. Using now that $\Omega$ is bounded, 
that $\partial\Omega$ is of class $C^2$, and the boundary $H^2$-regularity theorem from \cite[Theorem 4 in Section 6.3.2]{Evans2010}, we deduce that $v_K\in H^2(\Omega)^2$ and that there exists $C>0$ only depending on $\Omega$ such that
\begin{equation}\label{eq:norm_conv_reg_est_H2}
\|v_K\|_{H^2(\Omega)^2}
\leq C\big(\|v_K\|_{L^2(\Omega)^2}+\|\Delta v_K\|_{L^2(\Omega)^2}\big).
\end{equation}

The proof of the lemma finishes as follows.
On the one hand, since $v_K\in H^2(\Omega)^2$ and
$-i\sigma\cdot\nabla v_K=v$, we get $v\in H^1(\Omega)^2$ and $\|v\|_{H^1(\Omega)^2}\leq C\|v_K\|_{H^2(\Omega)^2}$.
On the other hand, \eqref{eq:vk_H1_v_L2} leads to
$\|v_K\|_{L^2(\Omega)^2}\leq C\|v\|_{L^2(\Omega)^2}$. With these estimates in hand, and recalling that 
$-\Delta v_K=-i\sigma\cdot\nabla v$, from 
\eqref{eq:norm_conv_reg_est_H2} we conclude that
\begin{equation}
\|v\|_{H^1(\Omega)^2}\leq C\|v_K\|_{H^2(\Omega)^2}
\leq C\big(\|v_K\|_{L^2(\Omega)^2}+\|\Delta v_K\|_{L^2(\Omega)^2}\big)
\leq C\big(\|v\|_{L^2(\Omega)^2}+\|\sigma\cdot\nabla v\|_{L^2(\Omega)^2}\big)
\end{equation}
for some $C>0$ depending only on $\Omega$.
\end{proof}

For our proof of \Cref{thm:res_conv_2}, apart from \Cref{Lemma:AlphaGradEstimates_ort}, we will also use the following result on traces, which is a restatement of \cite[Proposition 2.1]{Vega2016}.

\begin{lemma}\label{Lemma:SigmaGradEstimates}
Let $v \in L^2(\Omega)^2$ be such that $ \sigma \cdot \nabla v \in L^2(\Omega)^2$. Then, the trace of $v$ belongs to $H^{-1/2}(\partial \Omega)^2$ and satisfies 
$$\|v \|_{H^{-1/2}(\partial \Omega)^2} \leq C (   \| v \|_{L^2(\Omega)^2}    + \| \sigma \cdot \nabla v \|_{L^2(\Omega)^2} )$$
for some $C>0$ depending only on $\Omega$.
\end{lemma}

We finally have all the ingredients to prove  \Cref{thm:res_conv_2}.

\begin{proof}[Proof of \Cref{thm:res_conv_2}] 
Given $\lambda\in\C\setminus\R$ and $f,g\in L^2(\Omega)^4$, set 
\begin{equation}
\begin{split}
&\psi_\tau =(u_\tau, v_\tau)^\intercal:= (\Dirac_\tau-\lambda)^{-1}f \in \operatorname{Dom}(\Dirac_\tau),\\
&\varphi =(u, v)^\intercal := (\Dirac_{+\infty}-\overline\lambda)^{-1}P_+g \in 
\operatorname{Dom}(\Dirac_{+\infty}).
\end{split}
\end{equation}
Let us prove that $\varphi$ also belongs to  
$\ker(\Dirac_{+\infty}+ m)^{\perp}$ and to $H^1(\Omega)^4$. Given
$\psi \in \ker(\Dirac_{+\infty}+ m)$, since both $\varphi$ and $\psi$ belong to $\operatorname{Dom}(\Dirac_{+\infty})$ and 
$\Dirac_{+\infty}$ is self-adjoint in $L^2(\Omega)^4$, we see that
\begin{equation}
\begin{split}
-(m+\overline\lambda)\langle\varphi,\psi\rangle_{L^2(\Omega)^4}
&=\langle-\overline\lambda\varphi,\psi\rangle_{L^2(\Omega)^4}
+\langle\varphi,-m\psi\rangle_{L^2(\Omega)^4}\\
&=\langle-\overline\lambda\varphi,\psi\rangle_{L^2(\Omega)^4}
+\langle\varphi,\Dirac_{+\infty}\psi\rangle_{L^2(\Omega)^4}\\
&=\langle-\overline\lambda\varphi,\psi\rangle_{L^2(\Omega)^4}
+\langle\Dirac_{+\infty}\varphi,\psi\rangle_{L^2(\Omega)^4}\\
&=\langle(\Dirac_{+\infty}-\overline\lambda)\varphi,\psi\rangle_{L^2(\Omega)^4}\\
&=\langle P_+g,\psi\rangle_{L^2(\Omega)^4}.
\end{split}
\end{equation}
Recall that $P_+$ is a projection onto 
$\ker(\Dirac_{+\infty}+ m)^{\perp}$ and that $\psi \in \ker(\Dirac_{+\infty}+ m)$. Thus,
$\langle P_+g,\psi\rangle_{L^2(\Omega)^4}=0$ and, since 
$m+\overline\lambda\neq0$, we deduce that
$\langle\varphi,\psi\rangle_{L^2(\Omega)^4}=0$.
In conclusion, $\varphi \in \operatorname{Dom}(\Dirac_{+\infty})\cap \ker(\Dirac_{+\infty}+ m)^{\perp}$. Then, thanks to \Cref{Lemma:AlphaGradEstimates_ort}, we also get $\varphi\in H^1(\Omega)^4$ (and, in particular, $v\in H^1(\Omega)^2$).

Denote 
$$W_{\tau}:=P_{+}\big((\Dirac_{+\infty}-\lambda)^{-1} - (\Dirac_\tau-\lambda)^{-1}\big).$$
Using that $P_+$ is self-adjoint and integration by parts, we have
\begin{equation}\label{eq:density_delta_weak_5}
\begin{split}
\langle W_\tau f,g\rangle_{L^2(\Omega)^4}
&=\langle(\Dirac_{+\infty}-\lambda)^{-1}f - (\Dirac_\tau-\lambda)^{-1}f, P_+g \rangle_{L^2(\Omega)^4}\\
& = \langle f, (\Dirac_{+\infty}-\overline\lambda)^{-1}P_+g \rangle_{L^2(\Omega)^4} - \langle (\Dirac_\tau-\lambda)^{-1}f, P_+g \rangle_{L^2(\Omega)^4} \\
& = \langle (\Dirac_\tau-\lambda)\psi_\tau, \varphi \rangle_{L^2(\Omega)^4} 
- \langle \psi_\tau, (\Dirac_{+\infty}-\overline\lambda)\varphi \rangle_{L^2(\Omega)^4} \\
& = -i\langle \alpha\cdot\nabla \psi_\tau, \varphi \rangle_{L^2(\Omega)^4} -i \langle \psi_\tau, \alpha\cdot\nabla \varphi \rangle_{L^2(\Omega)^4} \\
&=-i\langle(\alpha\cdot\nu) \psi_\tau,\varphi\rangle_{L^2(\partial\Omega)^4}.
\end{split}
\end{equation}
Since $\psi_\tau \in \operatorname{Dom}(\Dirac_\tau)$, it holds that $\psi_\tau = i(\sinh \tau -\cosh \tau\, \beta)(\alpha \cdot \nu) \psi_\tau$ on $\partial \Omega$. This leads to $-i(\alpha \cdot \nu) \psi_\tau = (\sinh \tau +\cosh \tau\, \beta) \psi_\tau$ on $\partial \Omega$, and then \eqref{eq:density_delta_weak_5} yields
\begin{equation}
\begin{split}
\langle W_\tau f,g\rangle_{L^2(\Omega)^4}
=\langle (\sinh \tau +\cosh \tau\, \beta) \psi_\tau,\varphi\rangle_{L^2(\partial\Omega)^4}.
\end{split}
\end{equation}
Using now that $\varphi = (u, v)^\intercal$ with 
$u\in H^1_0(\Omega)^2$, which means that $u=0$ on 
$\partial\Omega$, and that $\sinh \tau -\cosh \tau = - e^{-\tau}$, we get
\begin{equation}
\begin{split}
\langle W_\tau f,g\rangle_{L^2(\Omega)^4}
=-e^{-\tau}\langle v_\tau,v\rangle_{L^2(\partial\Omega)^2}.
\end{split}
\end{equation}
Since both $v_\tau$ and $v$ belong to $H^1(\Omega)^2$,
by \eqref{eq:dual_pairing_L2} we have
$$\overline{\langle v_\tau,v \rangle _{L^2(\partial\Omega)^2}}
=\langle v,v_\tau \rangle _{L^2(\partial\Omega)^2} = \langle v_\tau, v \rangle_{H^{-1/2}(\partial\Omega)^2, H^{1/2}(\partial\Omega)^2}.$$ 
Then, using the trace theorem from $H^{1}(\Omega)^2$ to $H^{1/2}(\partial\Omega)^2$, \Cref{Lemma:SigmaGradEstimates} applied to $v_\tau$, and  \Cref{Lemma:AlphaGradEstimates_ort} applied to $\varphi$, we deduce that
\begin{equation}\label{eq:density_delta_weak_6}
\begin{split}
|\langle W_\tau f,g\rangle_{L^2(\Omega)^4}|
&=e^{-\tau}|\langle v_\tau,v\rangle_{L^2(\partial\Omega)^2}|\\
&=e^{-\tau}|\langle v_\tau,v\rangle_{H^{-1/2}(\partial\Omega)^2, H^{1/2}(\partial\Omega)^2}|\\
&\leq C e^{-\tau}\|v_\tau\|_{H^{-1/2}(\partial\Omega)^2}
\|\varphi\|_{H^{1}(\Omega)^4}\\
&\leq Ce^{-\tau}
 \big(   \| \psi_\tau \|_{L^2(\Omega)^4}    + \| \alpha \cdot \nabla \psi_\tau \|_{L^2(\Omega)^4} \big)
\big(   \| \varphi \|_{L^2(\Omega)^4}    + \| \alpha \cdot \nabla \varphi \|_{L^2(\Omega)^4} \big)
\end{split}
\end{equation}
for some $C>0$ depending only on $\Omega$.

Next, note that if $A:\operatorname{Dom}(A)\subset L^2(\Omega)^4\to L^2(\Omega)^4$ is a self-adjoint operator then 
\begin{equation}\label{eq:weak_bdd_inv}
\begin{split}
\|(A-\lambda)h\|^2_{L^2(\Omega)^4}
=\|(A-\re(\lambda))h\|^2_{L^2(\Omega)^4}+\im(\lambda)^2\|h\|^2_{L^2(\Omega)^4}
\geq\im(\lambda)^2\|h\|^2_{L^2(\Omega)^4}
\end{split}
\end{equation}
for all $h\in\operatorname{Dom}(A)$,
which yields $\|(A-\lambda)^{-1}\|_{L^2(\Omega)^4\to L^2(\Omega)^4}\leq1/|\!\im(\lambda)|$.
Since $\psi_\tau=(\Dirac_\tau-\lambda)^{-1}f$, by the triangle inequality and \eqref{eq:weak_bdd_inv} applied to $A=\Dirac_\tau$ and $h=\psi_\tau$, we see that
\begin{equation} 
\begin{split} \label{eq:bound_norm_DiracSobolev}
\| \psi_\tau \|_{L^2(\Omega)^4}+\| \alpha \cdot \nabla \psi_\tau \|_{L^2(\Omega)^4}
&\leq\| (-i\alpha \cdot \nabla+m\beta-\lambda) \psi_\tau \|_{L^2(\Omega)^4}+
(1+m+|\lambda|)\| \psi_\tau \|_{L^2(\Omega)^4}\\
&=\| (\Dirac_\tau-\lambda) \psi_\tau \|_{L^2(\Omega)^4}+
(1+m+|\lambda|)\| \psi_\tau \|_{L^2(\Omega)^4}\\
&\leq\Big(1+\frac{1+m+|\lambda|}{|\!\im(\lambda)|}\Big)
\| f \|_{L^2(\Omega)^4}.
\end{split}
\end{equation}
Similarly, since $\varphi = (\Dirac_{+\infty}-\overline\lambda)^{-1}P_+g$, by \eqref{eq:weak_bdd_inv} applied to 
$A=\Dirac_{+\infty}$ and $h=\varphi$, we get
\begin{equation}
\begin{split}
\| \varphi \|_{L^2(\Omega)^4}+\| \alpha \cdot \nabla \varphi \|_{L^2(\Omega)^4}
\leq\Big(1+\frac{1+m+|\lambda|}{|\!\im(\lambda)|}\Big)
\| P_+g \|_{L^2(\Omega)^4}
\leq\Big(1+\frac{1+m+|\lambda|}{|\!\im(\lambda)|}\Big)
\| g \|_{L^2(\Omega)^4},
\end{split}
\end{equation}
where we also used in the last inequality that 
$\|P_{+}\|_{L^2(\Omega)^4\to L^2(\Omega)^4}=1$. 
Plugging these two estimates into \eqref{eq:density_delta_weak_6} yields 
\begin{equation}
\begin{split}
|\langle W_\tau f,g\rangle_{L^2(\Omega)^4}|
\leq Ce^{-\tau}
 \Big(1+\frac{1+m+|\lambda|}{|\!\im(\lambda)|}\Big)^2
 \| f \|_{L^2(\Omega)^4}\| g \|_{L^2(\Omega)^4}.
\end{split}
\end{equation}
If we now divide both sides of this last inequality by 
$\| f \|_{L^2(\Omega)^4}\| g \|_{L^2(\Omega)^4}$ and we take the supremum among all functions  
$f,g\in L^2(\Omega)^4\setminus\{0\}$, we conclude that
\begin{equation}
\|W_\tau\|_{L^2(\Omega)^4 \to L^2(\Omega)^4}
\leq Ce^{-\tau}
\Big(1+\frac{1+m+|\lambda|}{|\!\im(\lambda)|}\Big)^2
\end{equation}
for some $C>0$ depending only on $\Omega$. The theorem follows by taking $\tau\uparrow+\infty$.
\end{proof}

\section{Convergence in the norm resolvent sense for finite values} \label{Sec:Continuity_Resolvent}

In this section we give a proof of \Cref{thm:res_conv_norm} based on regularity estimates, contrary to the proof given in \cite{Duran2024}, which is based on the analyticity in $\tau$ of the resolvent of $\Dirac_\tau$. We begin the section showing that, for every $\tau\in\R$,
 the resolvent of $\Dirac_\tau$ is a bounded operator from $L^2(\Omega)^4$ to $H^1(\Omega)^4$. This is a simple consequence of the self-adjointness of $\Dirac_\tau$ and the closed graph theorem, as the following lemma shows. 

\begin{lemma} \label{Corollary:compactResolvent}
    Given $\tau \in \R$, let $\Dirac_\tau$ be defined as in \eqref{def:Dirac_tau}. Then, for every $\lambda \in \C \setminus \R$, the resolvent $(\Dirac_\tau - \lambda)^{-1}$ is a bounded operator from $L^2(\Omega)^4$ to $H^1(\Omega)^4$. As a consequence, 
    $(\Dirac_\tau - \lambda)^{-1}$ is a compact operator from $L^2(\Omega)^4$ to $L^2(\Omega)^4$ and
    $\sigma(\Dirac_\tau)$ is purely discrete.
\end{lemma}

\begin{proof}
Assume for the moment that $(\Dirac_\tau - \lambda)^{-1}$ is a bounded operator from $L^2(\Omega)^4$ to $H^1(\Omega)^4$, that is to say, that there exists $C_\tau>0$ such that  
\begin{equation}\label{eq:closed_graph_thm_dirac}
\|(\Dirac_\tau - \lambda)^{-1}f\|_{H^1(\Omega)^4} \leq C_\tau\|f\|_{L^2(\Omega)^4}\quad\text{for all $f\in L^2(\Omega)^4$.}
\end{equation} 
Then, using that $H^1(\Omega)^4$ is compactly embedded in $L^2(\Omega)^4$ since $\Omega$ is bounded, we deduce that $(\Dirac_\tau - \lambda)^{-1}$ is a compact operator from $L^2(\Omega)^4$ to $L^2(\Omega)^4$. This together with \cite[Proposition 8.8 and the paragraph below it]{Taylor2011} yields the discreteness of 
$\sigma(\Dirac_\tau)$. It only remains to prove \eqref{eq:closed_graph_thm_dirac}.
 
By self-adjointness, we know that 
$(\Dirac_\tau - \lambda)^{-1}$ is an everywhere defined and bounded operator in $L^2(\Omega)^4$, and that
$(\Dirac_\tau - \lambda)^{-1}(L^2(\Omega)^4) =\operatorname{Dom}(\Dirac_\tau)\subset H^1(\Omega)^4$. Thanks to the closed graph theorem \cite[Theorem 2.9]{Brezis2011}, to prove \eqref{eq:closed_graph_thm_dirac} it is enough to check that the graph
$$G:=\{(f,\varphi)\in L^2(\Omega)^4\times H^1(\Omega)^4:\,
(\Dirac_\tau - \lambda)^{-1}f=\varphi\in\operatorname{Dom}(\Dirac_\tau)\}$$
is closed in the Banach space $E:=L^2(\Omega)^4\times H^1(\Omega)^4$. This last assertion is what we will prove now. 
Assume that $(f_n,\varphi_n)\in G$ tends to $(f,\varphi)\in E$ as $n\uparrow+\infty$. Then, $f_n\to f$ in $L^2(\Omega)^4$ and 
$\varphi_n\to\varphi$ in $H^1(\Omega)^4$ as $n\uparrow+\infty$, which leads to 
$$(\Dirac_\tau-\lambda)\varphi
=\lim_{n\uparrow+\infty}(\Dirac_\tau-\lambda)\varphi_n
=\lim_{n\uparrow+\infty}f_n=f\quad\text{in }L^2(\Omega)^4.$$
In addition, since $\varphi_n = i (\sinh\tau- \cosh\tau \, \beta)( \alpha \cdot\nu ) \varphi_n$ in $L^2(\partial\Omega)^4$ for all $n$, by the trace theorem and the fact that 
$\varphi_n\to\varphi$ in $H^1(\Omega)^4$ as $n\uparrow+\infty$ we get that $\varphi= i (\sinh\tau- \cosh\tau \, \beta)( \alpha \cdot\nu ) \varphi$ in $L^2(\partial\Omega)^4$. Therefore, $\varphi\in \operatorname{Dom}(\Dirac_\tau)$ and $(\Dirac_\tau-\lambda)\varphi=f$. This means that 
$\varphi=(\Dirac_\tau-\lambda)^{-1}f$ and, thus, that 
$(f,\varphi)\in G$. That is to say, $G$ is closed in $E$.
\end{proof}

\begin{remark}\label{rk:disc_spec}
In the proof of \cite[Lemma 1.2]{Mas2022} it is said that the discreteness of $\sigma(\Dirac_\tau)$ follows from the compact embedding of $H^1(\Omega)^4$ in $L^2(\Omega)^4$, and no more details are given there. This compact embedding yields compactness of the resolvent ---and, thus, discreteness of the spectrum--- once the boundedness of the resolvent from $L^2(\Omega)^4$ to $H^1(\Omega)^4$ is shown. Since this last detail is omitted in the proof of \cite[Lemma 1.2]{Mas2022}, for the sake of completeness we find convenient to address it in the present work. In the proof of \Cref{Corollary:compactResolvent} we provide the full justification (with elementary arguments) of all these assertions.
\end{remark}

From the proof of \Cref{Corollary:compactResolvent} it is not clear how the constant $C_\tau$ on the right hand side of \eqref{eq:closed_graph_thm_dirac} depends on $\tau$. In some occasions, this could be a drawback when addressing a spectral analysis of  $\Dirac_\tau$ in terms of $\tau$.
The following result, whose proof does not use \eqref{eq:closed_graph_thm_dirac}, is the quantitative counterpart of \Cref{Corollary:compactResolvent}. We think that it has its own interest, and it may be useful to present it here for future references.

\begin{theorem} \label{thm:ContinuityResolvent}
    Let $\tau \in \R$, $\Dirac_\tau$ be defined as in \eqref{def:Dirac_tau}, and $\lambda \in \C \setminus \R$. Then, there exists $C>0$ depending only on $\Omega$ such that
    \begin{equation}
        \|(\Dirac_\tau - \lambda)^{-1}f\|_{H^1(\Omega)^4} \leq C\cosh\tau \left( 1+\frac{1+m+|\lambda|}{|\!\im(\lambda)|} \right) \|f\|_{L^2(\Omega)^4}
        \quad\text{for all $f\in L^2(\Omega)^4$.}
    \end{equation}
\end{theorem}

In view of the bound given in \eqref{eq:bound_norm_DiracSobolev}, \Cref{thm:ContinuityResolvent} would follow if the inclusion of $\mathrm{Dom}(\Dirac_\tau)$ in $H^1(\Omega)^4$ was continuous with respect to the graph norm of $\Dirac_\tau$ (in the same spirit as \eqref{ineq:cont_graph_h1}) with the suitable quantitative control on $\tau$. The following lemma shows that this is the case.

\begin{lemma} \label{Lemma:boundH1normHtau}
Given $\tau \in \R$, let $\Dirac_\tau$ be defined as in \eqref{def:Dirac_tau}. Then,
    \begin{equation}
        \|\nabla \varphi\|_{L^2(\Omega)^4} \leq C \cosh\tau \left( \| \varphi\|_{L^2(\Omega)^4} + \|\alpha\cdot\nabla \varphi\|_{L^2(\Omega)^4} \right)
    \end{equation}
    for all $\varphi \in \mathrm{Dom}(\Dirac_\tau)$, where $C>0$ depends only on $\Omega$.
\end{lemma}

To prove this lemma we will first show it in the case $\tau=0$. Then, the general case $\tau\in\R$ will easily follow from a simple observation which relates 
$\mathrm{Dom(\Dirac_\tau)}$ and $\mathrm{Dom}(\Dirac_0)$, as stated in the following lemma. 

\begin{lemma} \label{Lemma:FromHtautoH0}
    Let $\tau\in\R$. Given $\varphi=(u,v)^\intercal\in L^2(\Omega)^4$ set
    $$ \psi := \begin{pmatrix}
            e^{\tau/2}u \\
            e^{-\tau/2}v
        \end{pmatrix}
        =\begin{pmatrix}
            e^{\tau/2} & 0 \\
            0 & e^{-\tau/2}
        \end{pmatrix} \varphi \in L^2(\Omega)^4. 
    $$
 Then,  $\varphi\in\mathrm{Dom}(\Dirac_\tau)$ if and only if 
 $\psi\in \mathrm{Dom}(\Dirac_0)$.
\end{lemma}

\begin{proof}
    Obviously, $\varphi \in H^1(\Omega)^4$ if and only if $\psi \in H^1(\Omega)^4$. Thus, we only need to take care of the boundary conditions. Assume first that 
    $\varphi\in\mathrm{Dom}(\Dirac_\tau)$. Since $\varphi=i(\sinh \tau - \cosh \tau \beta)(\alpha\cdot\nu) \varphi$ on $\partial\Omega$, from the definition of $\psi$ we see that
    \begin{equation}
        \begin{split}
            \psi & = \begin{psmallmatrix}
            e^{\tau/2} & 0 \\
            0 & e^{-\tau/2}
        \end{psmallmatrix} i(\sinh \tau - \cosh \tau \beta) (\alpha\cdot\nu) \varphi 
             = -i \beta \begin{psmallmatrix}
            e^{-\tau/2} & 0 \\
            0 & e^{\tau/2}
            \end{psmallmatrix} (\alpha\cdot\nu) \varphi \\
            & = -i \beta (\alpha\cdot\nu) \begin{psmallmatrix}
            e^{\tau/2} & 0 \\
            0 & e^{-\tau/2}
            \end{psmallmatrix} \varphi 
             = -i \beta (\alpha\cdot\nu) \psi
        \end{split}
    \end{equation}
    on $\partial\Omega$. Therefore, $\psi\in\mathrm{Dom}(\Dirac_0)$. The other implication is analogous.
\end{proof}

With this ingredient in hand, let us now address the proof of \Cref{Lemma:boundH1normHtau}.

\begin{proof}[Proof of \Cref{Lemma:boundH1normHtau}]
    We shall first prove the lemma for $\tau =0$. Assume that
    $\psi \in \mathrm{Dom}(\Dirac_0)$. By \cite[formula $(1.3)$]{Arrizabalaga2017}, we have
    \begin{equation}
        \|\nabla\psi\|_{L^2(\Omega)^4}^2 = \|\alpha \cdot \nabla\psi\|_{L^2(\Omega)^4}^2 - \dfrac{1}{2}\int_{\partial\Omega} \kappa |\psi|^2 \, d\upsigma,
    \end{equation}
    where $\kappa$ denotes the mean curvature of $\partial\Omega$. Using that $\Omega$ is a bounded domain with $C^2$ boundary, and \cite[Lemma 2.6]{Behrndt2014}, given  
    $\epsilon>0$ we see that
    \begin{equation}
        \begin{split}
            \|\nabla\psi\|_{L^2(\Omega)^4}^2 & \leq \|\alpha \cdot \nabla\psi\|_{L^2(\Omega)^4}^2 + \dfrac{1}{2} \|\kappa\|_{L^\infty(\partial\Omega)} \|\psi\|_{L^2(\partial \Omega)^4}^2 \\
            & \leq \|\alpha \cdot \nabla\psi\|_{L^2(\Omega)^4}^2 + \dfrac{1}{2} \|\kappa\|_{L^\infty(\partial\Omega)} \big( \epsilon \|\nabla\psi\|_{L^2(\Omega)^4}^2 + C_\epsilon \|\psi\|_{L^2(\Omega)^4}^2 \big)
        \end{split}
    \end{equation}
   for some $C_\epsilon>0$ depending only on $\epsilon$ and $\Omega$. Taking $\epsilon= \|\kappa\|_{L^\infty(\partial\Omega)}^{-1}$ and grouping terms, we conclude that
    \begin{equation} \label{eq:boundH1normMIT}
        \|\nabla\psi\|_{L^2(\Omega)^4}^2 \leq C \big( \|\psi\|_{L^2(\Omega)^4}^2 + \|\alpha \cdot \nabla\psi\|_{L^2(\Omega)^4}^2 \big)
    \end{equation}
    for some $C>0$ depending only on $\Omega$. This proves the lemma for $\tau=0$.\footnote{Alternatively, a proof for $\tau=0$  based on \eqref{eq:closed_graph_thm_dirac} can be carried out with no difficulties.} 
    
    We now prove the lemma for arbitrary $\tau \in \R$. Assume that $\varphi \in \mathrm{Dom}(\Dirac_\tau)$. By \Cref{Lemma:FromHtautoH0}, the function 
    \begin{equation}
        \psi := \begin{pmatrix}
            e^{\tau/2} & 0 \\
            0 & e^{-\tau/2}
        \end{pmatrix} \varphi
    \end{equation}
    belongs to $\mathrm{Dom}(\Dirac_0)$ and, therefore, \eqref{eq:boundH1normMIT} holds for this $\psi$. As a consequence,
    \begin{equation}
        \begin{split}
            \|\nabla\varphi\|_{L^2(\Omega)^4}^2 & = \left\| \nabla \begin{psmallmatrix}
            e^{-\tau/2} & 0 \\
            0 & e^{\tau/2}
        \end{psmallmatrix} \psi \right\|_{L^2(\Omega)^4}^2 
         \leq (e^\tau+e^{-\tau}) \|\nabla \psi\|_{L^2(\Omega)^4}^2 
         \\
        & \leq C\cosh\tau \left( \|\psi\|_{L^2(\Omega)^4}^2 + \|\alpha \cdot \nabla\psi\|_{L^2(\Omega)^4}^2 \right) \\
        & = C\cosh\tau \left( \left\|\begin{psmallmatrix}
            e^{\tau/2} & 0 \\
            0 & e^{-\tau/2}
        \end{psmallmatrix} \varphi \right\|_{L^2(\Omega)^4}^2 + \left\|\alpha \cdot \nabla\begin{psmallmatrix}
            e^{\tau/2} & 0 \\
            0 & e^{-\tau/2}
        \end{psmallmatrix} \varphi \right\|_{L^2(\Omega)^4}^2 \right) \\
        & \leq C\cosh^2\tau \left( \|\varphi\|_{L^2(\Omega)^4}^2 + \|\alpha \cdot \nabla\varphi\|_{L^2(\Omega)^4}^2 \right)
        \end{split}
    \end{equation}
for some $C>0$ depending only on $\Omega$.
\end{proof}

With these results, \Cref{thm:ContinuityResolvent} follows straightforwardly.

\begin{proof}[Proof of \Cref{thm:ContinuityResolvent}]
    Given $\lambda \in \C\setminus \R$ and $f\in L^2(\Omega)^4$, set $\varphi_\tau := (\Dirac_\tau - \lambda)^{-1}f \in \mathrm{Dom}(\Dirac_\tau)$. By \Cref{Lemma:boundH1normHtau} and arguing as in  \eqref{eq:bound_norm_DiracSobolev}, we see that
    \begin{equation}
        \begin{split}
            \|(\Dirac_\tau - \lambda)^{-1}f\|_{H^1(\Omega)^4} & \leq C\cosh\tau \left( \| \varphi_\tau\|_{L^2(\Omega)^4} + \|\alpha\cdot\nabla \varphi_\tau\|_{L^2(\Omega)^4} \right) \\ 
            & \leq C\cosh\tau \left( 1+\frac{1+m+|\lambda|}{|\!\im(\lambda)|} \right) \| f \|_{L^2(\Omega)^4}
        \end{split}
    \end{equation}
    for some $C>0$ depending only on $\Omega$,
    as desired.
\end{proof}

To conclude this section, we address the proof of \Cref{thm:res_conv_norm} based on estimating  \eqref{eq:pairing_tau0}.

\begin{proof}[Proof of \Cref{thm:res_conv_norm}]
    By \cite[Theorem VIII.19]{ReedSimon1980} it is enough to prove that the difference of resolvents at $\lambda=i$, namely
    \begin{equation}
        W_\tau := (\Dirac_{\tau_0}-i)^{-1} - (\Dirac_\tau - i)^{-1},
    \end{equation}
    converges to zero in norm as $\tau \to \tau_0$. Given $f,g\in L^2(\Omega)^4$, set
    \begin{equation}
        \begin{split}
            &\psi_\tau =(u_\tau, v_\tau)^\intercal:= (\Dirac_\tau-i)^{-1}f \in \operatorname{Dom}(\Dirac_\tau),\\
            &\varphi =(u, v)^\intercal := (\Dirac_{\tau_0}+i)^{-1}g \in \operatorname{Dom}(\Dirac_{\tau_0}).
        \end{split}
    \end{equation}
    Integration by parts leads to
    \begin{equation} 
        \begin{split}
            \langle W_\tau f, g \rangle_{L^2(\Omega)^4} & = \langle(\Dirac_{\tau_0}-i)^{-1}f - (\Dirac_\tau-i)^{-1}f, g \rangle_{L^2(\Omega)^4} \\
            & = \langle f, (\Dirac_{\tau_0}+i)^{-1}g \rangle_{L^2(\Omega)^4} - \langle (\Dirac_\tau-i)^{-1}f, g \rangle_{L^2(\Omega)^4} \\
            & = \langle (\Dirac_\tau-i)\psi_\tau, \varphi \rangle_{L^2(\Omega)^4} - \langle \psi_\tau, (\Dirac_{\tau_0}+i)\varphi \rangle_{L^2(\Omega)^4} \\
            & = -i\langle \alpha\cdot\nabla \psi_\tau, \varphi \rangle_{L^2(\Omega)^4} -i \langle \psi_\tau, \alpha\cdot\nabla \varphi \rangle_{L^2(\Omega)^4} \\
            & = -i\langle(\alpha\cdot\nu) \psi_\tau, \varphi \rangle_{L^2(\partial\Omega)^4} \\
            & = -i \int_{\partial\Omega} 
            \big((\sigma\cdot\nu) 
              v_\tau\cdot \overline u + (\sigma\cdot\nu)u_\tau\cdot \overline v \big) \, d\upsigma.
        \end{split}
    \end{equation}
    Since $\psi_\tau\in \mathrm{Dom}(\Dirac_\tau)$ and $\varphi\in \mathrm{Dom}(\Dirac_{\tau_0})$, it holds that $v_\tau = ie^\tau (\sigma\cdot\nu) u_\tau$ and $v = ie^{\tau_0} (\sigma\cdot\nu) u$ on $\partial \Omega$. Hence, the previous computation leads to
    \begin{equation}
        \left( e^{-\tau_0} - e^{-\tau} \right) \int_{\partial\Omega} v_\tau \cdot \overline v \, d\upsigma = \langle W_\tau f, g \rangle_{L^2(\Omega)^4} = \left( e^\tau - e^{\tau_0} \right) \int_{\partial\Omega} u_\tau \cdot \overline u \, d\upsigma.
    \end{equation}
    Using this and the fact that $\langle \psi_\tau, \varphi \rangle_{L^2(\partial\Omega)^4} = \int_{\partial\Omega} u_\tau \cdot \overline u \, d\upsigma + \int_{\partial\Omega} v_\tau \cdot\overline v \, d\upsigma$, we see that
    \begin{equation}
        \langle \psi_\tau, \varphi \rangle_{L^2(\partial\Omega)^4} = \dfrac{1}{e^\tau - e^{\tau_0}} \langle W_\tau f, g \rangle_{L^2(\Omega)^4} + \dfrac{1}{e^{-\tau_0} - e^{-\tau}} \langle W_\tau f, g \rangle_{L^2(\Omega)^4},
    \end{equation}
    from where we conclude that
    \begin{equation}\label{eq:norm_conv_tau0_last2}
        \langle W_\tau f, g \rangle_{L^2(\Omega)^4} = \dfrac{\sinh \frac{\tau-\tau_0}{2}}{\cosh \frac{\tau+\tau_0}{2}} \langle \psi_\tau, \varphi \rangle_{L^2(\partial\Omega)^4}.
    \end{equation}
    
    Since both $\psi_\tau$ and $\varphi$ belong to $H^1(\Omega)^4$, by \eqref{eq:dual_pairing_L2}, the trace theorem from $H^1(\Omega)^4$ to $H^{1/2}(\Omega)^4$, and \Cref{Lemma:SigmaGradEstimates} applied to both of the components of $\varphi=(u,v)^\intercal$, we deduce that
    \begin{equation}\label{eq:norm_conv_tau0_last}
        \begin{split}
            |\langle W_\tau f, g \rangle_{L^2(\Omega)^4}| & = \left| \dfrac{\sinh \frac{\tau-\tau_0}{2}}{\cosh \frac{\tau+\tau_0}{2}} \right| \left| \langle \varphi, \psi_\tau \rangle_{H^{-1/2}(\partial\Omega)^4, H^{1/2}(\partial\Omega)^4} \right| \\
            & \leq C \left| \dfrac{\sinh \frac{\tau-\tau_0}{2}}{\cosh \frac{\tau+\tau_0}{2}} \right| \|\varphi\|_{H^{-1/2}(\partial\Omega)^4} \|\psi_\tau\|_{H^1(\Omega)^4} \\
            & \leq C \left| \dfrac{\sinh \frac{\tau-\tau_0}{2}}{\cosh \frac{\tau+\tau_0}{2}} \right| \left( \|\varphi\|_{L^2(\Omega)^4} + \|\alpha\cdot\nabla \varphi\|_{L^2(\Omega)^4} \right) \|(\Dirac_\tau-i)^{-1}f\|_{H^1(\Omega)^4}
        \end{split}
    \end{equation}
   for some $C>0$ depending only on $\Omega$. Now, arguing as in \eqref{eq:bound_norm_DiracSobolev} but on $\varphi$ and $g$, and using also \Cref{thm:ContinuityResolvent}, the previous estimate leads to
    \begin{equation}
         |\langle W_\tau f, g \rangle_{L^2(\Omega)^4}| \leq C_m \left| \dfrac{\sinh \frac{\tau-\tau_0}{2}}{\cosh \frac{\tau+\tau_0}{2}} \right| \cosh\tau \|f\|_{L^2(\Omega)^4} \|g\|_{L^2(\Omega)^4},
    \end{equation}
    for some $C_m>0$ depending only on $m$ and $\Omega$. Dividing both sides of this last inequality by $\|f\|_{L^2(\Omega)^4} \|g\|_{L^2(\Omega)^4}$ and taking the supremum among all functions $f,g \in L^2(\Omega)^4\setminus\{0\}$, we conclude that
    \begin{equation}
        \|W_\tau\|_{L^2(\Omega)^4\rightarrow L^2(\Omega)^4} \leq C_m \left| \dfrac{\sinh \frac{\tau-\tau_0}{2}}{\cosh \frac{\tau+\tau_0}{2}} \right| \cosh\tau
    \end{equation}
    for some $C_m>0$ depending only on $m$ and $\Omega$. The theorem follows by taking $\tau\to \tau_0$.
    
A final comment is in order. The proof of this theorem can also be carried out using \eqref{eq:closed_graph_thm_dirac} instead of \Cref{thm:ContinuityResolvent}. To do it, in view of \eqref{eq:norm_conv_tau0_last2}, one simply 
 replaces $$\left| \langle \varphi, \psi_\tau \rangle_{H^{-1/2}(\partial\Omega)^4, H^{1/2}(\partial\Omega)^4} \right|\quad\text{by}\quad
 \left| \langle \psi_\tau, \varphi  \rangle_{H^{-1/2}(\partial\Omega)^4, H^{1/2}(\partial\Omega)^4} \right|$$
in \eqref{eq:norm_conv_tau0_last}. In this way, when arguing as in \eqref{eq:norm_conv_tau0_last}, the term 
$\|\psi_\tau\|_{H^{-1/2}(\partial\Omega)^4}$ can be estimated independently of $\tau$ thanks to \Cref{Lemma:SigmaGradEstimates} and \eqref{eq:bound_norm_DiracSobolev}, and the term 
$\|\varphi\|_{H^1(\Omega)^4}$ can be estimated using \eqref{eq:closed_graph_thm_dirac} for $\tau=\tau_0$. 
\end{proof}

\end{document}